\documentclass[reqno,twoside]{article}

\usepackage{amssymb}
\usepackage{amsmath}
\usepackage{amsthm}
\usepackage{letterswitharrows}
\usepackage{mathrsfs}
\usepackage{graphicx}
\usepackage{latexsym}
\usepackage{stmaryrd}
\usepackage{enumitem}
\setlist{topsep=0.3em, itemsep=-0.3em}
\usepackage{moreenum}
\usepackage{authblk}
\usepackage[bookmarks,hyperfootnotes=false, psdextra=true]{hyperref}
\usepackage{multirow}
\usepackage[dvipsnames]{xcolor}
\usepackage{caption}
\colorlet{darkishRed}{red!60!black}
\colorlet{darkishBlue}{blue!60!black}
\colorlet{darkishGreen}{green!50!black}
\colorlet{lightishGreen}{green!70!black}
\hypersetup{
    draft = false,
    bookmarksopen=true,
    colorlinks,
    linkcolor={darkishBlue},
    citecolor={lightishGreen},
    urlcolor={darkishBlue}
}
\usepackage[nameinlink, capitalise, noabbrev]{cleveref}
\usepackage{thmtools}
\usepackage{nameref}
\crefformat{enumi}{#2#1#3}
\crefformat{equation}{#2(#1)#3}
\crefrangeformat{section}{Sections~#3#1#4--#5#2#6}
\crefname{mainresult}{Theorem}{Theorems}
\crefname{maincorollary}{Corollary}{Corollaries}
\let\setminus=\smallsetminus
\usepackage{tikz}
\usepackage{tikz-cd}
\usetikzlibrary{calc,through,intersections,arrows, trees, positioning, decorations.pathmorphing, cd}
\usepackage{comment}
\usepackage{mathtools}
\usepackage{nccmath}
\usepackage{pifont}
\usepackage[utf8]{inputenc}
\usepackage[T1]{fontenc}
\usepackage{lmodern}
\usepackage[babel]{microtype}
\usepackage[english]{babel}
\usepackage{relsize}

\usepackage{caption}

\usetikzlibrary{decorations.markings,arrows.meta}

\usepackage{geometry}
\let\setminus=\smallsetminus
\renewcommand{\leq}{\leqslant}
\renewcommand{\geq}{\geqslant}
\renewcommand{\ge}{\geq}
\renewcommand{\le}{\leq}

\let\rho=\varrho
\let\phi=\varphi

\newcommand{ \N } { \mathbb{N} }

\newcommand{\defn}[1]{{\color{darkishRed}{\emph{#1}}}}
\newcommand{\defnm}[1]{{\color{darkishRed}{#1}}}

\makeatletter

\def\calCommandfactory#1{%
   \expandafter\def\csname c#1\endcsname{\mathcal{#1}}}
\def\frakCommandfactory#1{%
   \expandafter\def\csname frak#1\endcsname{\mathfrak{#1}}}

\newcounter{ctr}
\loop
  \stepcounter{ctr}
  \edef\X{\@Alph\c@ctr}
  \expandafter\calCommandfactory\X
  \expandafter\frakCommandfactory\X
  \edef\Y{\@alph\c@ctr}
  \expandafter\frakCommandfactory\Y
\ifnum\thectr<26
\repeat

\setenumerate{label={\normalfont (\roman*)}}

\newtheorem{theorem}{Theorem}[section] 
\newtheorem{corollary}[theorem]{Corollary}
\newtheorem{lemma}[theorem]{Lemma}

\newtheorem{observation}[theorem]{Observation}

\newtheorem{problem}[theorem]{Problem}
\newtheorem{mainresult}{Theorem}
\newtheorem{maincorollary}[mainresult]{Corollary}
\newtheorem{mainconjecture}[mainresult]{Conjecture}

\newtheorem{claim}{Claim}
\crefname{claim}{Claim}{Claims}
\AtEndEnvironment{proof}{\setcounter{claim}{0}}

\newenvironment{claimproof}{\noindent\textit{Proof.}}{\hfill\ensuremath{\blacksquare}\medskip}
\usepackage{etoolbox}

\theoremstyle{definition}

\newtheorem{definition}[theorem]{Definition}

\theoremstyle{remark}

\usepackage[skip=6pt]{subcaption} 
\usepackage{etoolbox}
\newbool{pdfBool}
\booltrue{pdfBool} 

\newcommand{\pdfOrNot}[2]{\ifbool{pdfBool}{{#1}}{{#2}}}

\usepackage{svg}

\newbool{arXiv}
\booltrue{arXiv} 

\newcommand{\arXivOrNot}[2]{\ifbool{arXiv}{{#1}}{{#2}}}

\usepackage{csquotes}
\usepackage[backend=biber, style=alphabetic, sorting=nyt, maxnames=99, maxalphanames=99, giveninits=true, sortcites=true]{biblatex}
\usepackage{authblk}

\usepackage{amsmath}
\usepackage{fancyhdr}
\newcommand{\bivec}[1]{\overset{\scriptscriptstyle\leftrightarrow}{#1}}

\newsavebox{\otterbox}
\sbox{\otterbox}{\includegraphics[height=1em]{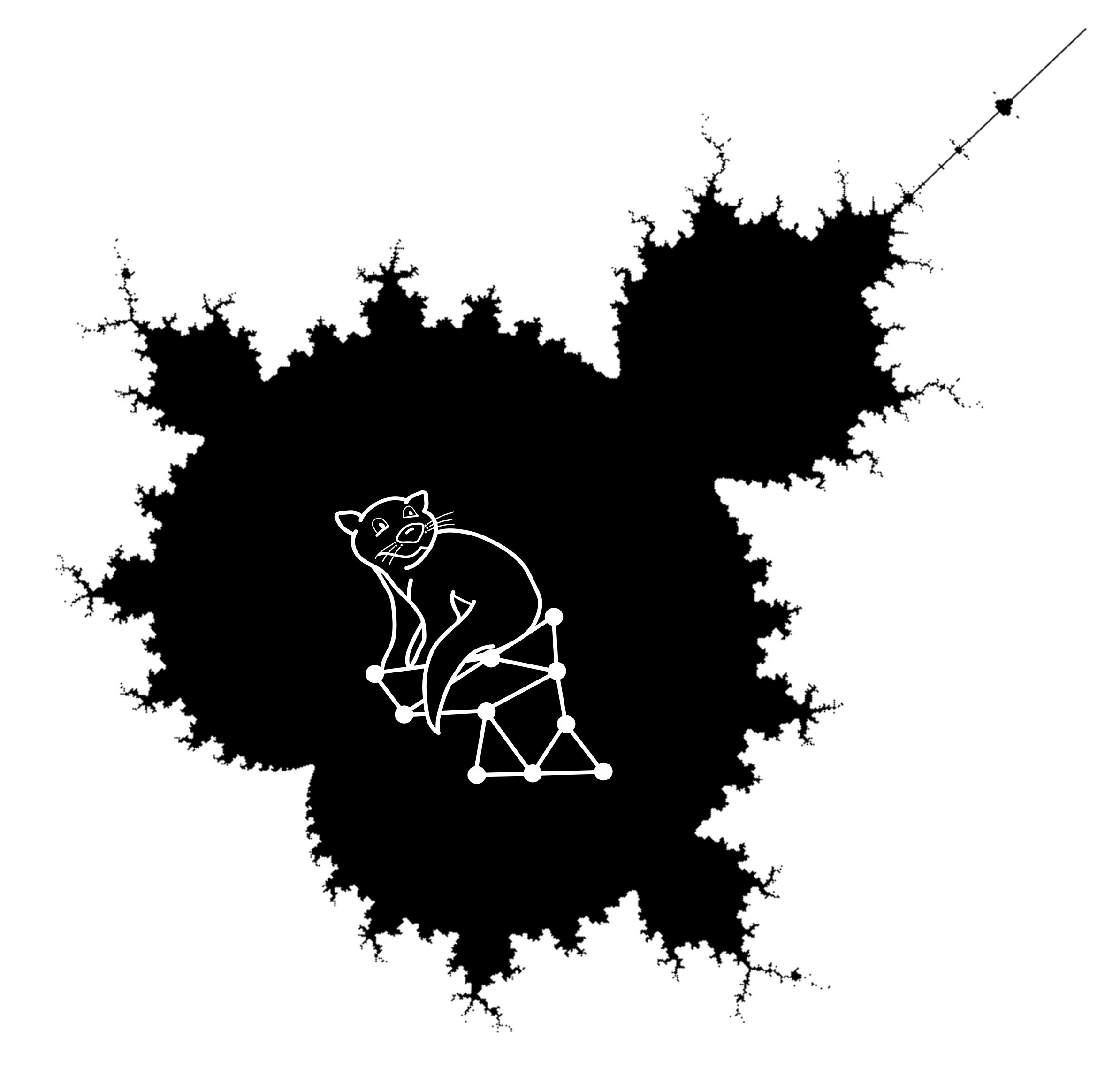}}

\makeatletter

\renewcommand{\@fnsymbol}[1]{%
  \ifcase#1\or
    \raisebox{-0.3ex}{\usebox{\otterbox}}%
  \or
    $\dagger$%
  \or
    $\ddagger$%
  \or
    \S
  \or
    \P
  \else
    *
  \fi
}

\makeatother

\title{Proof of Lichiardopol's conjecture on\\ disjoint directed cycles of distinct lengths}

\author[1]{Sandra Albrechtsen\thanks{Supported by the Alexander von Humboldt Foundation in the framework of the Alexander von Humboldt Professorship of Daniel Král' endowed by the Federal Ministry of Education and Research.}}

\affil[1]{Institute of Mathematics, Leipzig University, Augustusplatz 10, 04109 Leipzig, Germany\\ sandra@albrechtsen-mail.de}

\author[2]{Raphael Steiner\thanks{Research supported by the SNSF Ambizione Grant No. 216071 of the Swiss National Science Foundation.}}

\affil[2]{Department of Mathematics, ETH Zurich, R\"{a}mistrasse 101, 8092 Z\"{u}rich, Switzerland\\
raphaelmario.steiner@math.ethz.ch}

\begin{document}

\date{\vspace{-5ex}}

\maketitle

\begin{abstract}
There is a fascinating array of interrelated questions studying which structures can be guaranteed in digraphs of large minimum out-degree. These often have intriguingly simple statements, yet seem surprisingly difficult to approach. A well-known example is Lichiardopol's conjecture (2014), stating that there exists a function $g:\mathbb{N}\rightarrow \mathbb{N}$ such that every digraph with minimum out-degree at least $g(k)$ contains $k$ vertex-disjoint directed cycles of distinct lengths. 

In this paper, building on earlier work of the second author, we confirm this conjecture in full generality. We also generalise this result to a weighted setting. Our proof uses and combines many ingredients from structural digraph theory such as butterfly minors, directed tangles, a directed analogue of the Tangle-Wall Theorem due to Robertson and Seymour as well as a local variant of the Directed Flat Wall Theorem due to Giannopoulou, Kawarabayashi, Kreutzer and Kwon. These techniques, which are somewhat atypical in the study of minimum degree conditions, may be of independent interest and may find further applications.
\end{abstract}


\section{Introduction}

Cycles are amongst the most fundamental and ubiquitous objects in graph theory. Therefore, the study of sufficient conditions for the existence of cycles (or collections of cycles) in graphs with various restrictions, such as on their intersections and (relative) lengths, has a long and rich history, but nevertheless remains very active today. A particular type of question in this line of research, which has attracted a lot of attention over the last half century, is the study of minimum degree conditions for the existence of \emph{disjoint cycles} in graphs (often with various additional properties), see~\cite{chibasurvey} for a survey of the large body of work on this topic. Generally speaking, for \emph{undirected} graphs many of the most natural and fundamental problems in this area have been answered in the past, for example the best-possible minimum degree conditions have been determined guaranteeing 
\begin{itemize}
    \item $k$ vertex-disjoint cycles~\cite{corradihajnal},
    \item $k$ vertex-disjoint cycles of equal lengths~\cite{alon,egawa,haggkvist},
    \item $k$ vertex-disjoint cycles of pairwise distinct lengths~\cite{bensmail}, and
    \item $k$ vertex-disjoint cycles of even lengths~\cite{chibaeven}.
\end{itemize}

In contrast, comparatively little is known about the analogous problems for \emph{directed} cycles in \emph{directed} graphs (digraphs for short), where even rather fundamental questions remain open. Note that in the directed setting, it is natural to require the minimum \emph{out-}degree $\delta^+(D)$ of the digraph $D$ in question to be large, rather than the minimum total degree, since there are very dense digraphs (such as transitive tournaments) that do not even contain a single directed cycle. Perhaps the most fundamental question is then to determine the smallest possible function $b:\mathbb{N}\rightarrow \mathbb{N}$ such that every digraph of minimum out-degree at least $b(k)$ contains $k$ vertex-disjoint directed cycles. Note that even the existence of such a function $b$ is not obvious at all, and this was first proved using a clever contraction trick in a seminal result by Thomassen~\cite{ThomassenDisjointCycles}. In the same paper, Thomassen also conjectured that the precise bound $b(k)=2k-1$ holds for every $k$; the same conjecture in fact appeared already earlier in the survey article~\cite{bermondthomassen} on cycles in digraphs by Bermond and Thomassen in 1981, and is nowadays famous as the (still open) \emph{Bermond-Thomassen conjecture}. See \cite{k3,alon,bucic} for some of the most important partial results towards this well-known conjecture.

\smallskip

Even less is known about the existence of vertex-disjoint directed cycles with additional properties. As previously mentioned, it is known that sufficiently high minimum degree forces many disjoint cycles of equal length in \emph{undirected} graphs, and Thomassen~\cite{ThomassenDisjointCycles} conjectured that the analogous statement is true for digraphs. Surprisingly, this was shown to be false by Alon~\cite{alon}, who constructed digraphs of arbitrarily high minimum out-degree which do not even contain two edge-disjoint cycles of the same length; the second author~\cite{rapha} later strengthened this by showing that even arbitrarily high strong vertex-connectivity does not guarantee two edge-disjoint directed cycles of the same length. 

In the opposite direction, Henning and Yeo~\cite{henningyeo} initiated the study of the problem of forcing vertex-disjoint directed cycles of \emph{distinct lengths}, and conjectured that every digraph with minimum out-degree at least~$4$ contains two vertex-disjoint directed cycles of distinct lengths. This conjectured bound (which is best possible) was confirmed by Lichiardopol~\cite{Lich14}. Lichiardopol then raised the natural question about packings of more directed cycles of distinct lengths, and made the following intriguing conjecture. 

\begin{mainconjecture}[{Lichiardopol \cite{Lich14}}] \label{con:lichiardopol}
    There exists a function $g: \N \to \N$ such that for every $k \in \N$ every digraph~$D$ with $\delta^+(D)\ge g(k)$ contains $k$ vertex-disjoint directed cycles of pairwise distinct lengths. 
\end{mainconjecture}

This conjecture has attracted a significant amount of attention over the last 12 years and has been reiterated several times, see ~\cite[Conjecture~26]{bensmail},~\cite[Conjecture~2.5.25]{bangjensen},~\cite[Conjecture~2]{rapha},~\cite[Conjecture~1.2]{H1.3},~\cite[Conjecture~1.2]{D1.2},~\cite[Conjecture~6.3]{gollin}. In a recent paper~\cite{gollin}, Gollin, Gorsky, Hatzel, Hendrey, Huynh, McFarland, Soko{\l}owski, Wiederrecht and Wollan describe the conjecture as `famously open'. 

Despite this attention, the conjecture has remained unresolved even for the small case $k=3$. Among noteworthy partial results, Bensmail, Harutyunyan, Le, Li and Lichiardopol~\cite{bensmail} proved the conjecture for tournaments and regular digraphs. 
Chen and Chang~\cite{D1.2} proved it for bipartite tournaments, Song and Yan~\cite{H1.3} proved it for semicomplete digraphs and digraphs where the maximum in-degree is upper-bounded by an exponential function of the minimum out-degree, and the second author~\cite{rapha} proved it for digraphs of bounded directed tree-width as well as for digraphs of sufficiently large strong vertex-connectivity.
See also~\cite{A,B,C,E,F,G} for further work on disjoint directed cycles of different lengths under out-degree conditions.

The main result of this paper is to settle Lichiardopol's conjecture in full generality.

\begin{mainresult} \label{main:CyclesDistinctLengths}
    There exists a function $g: \N \to \N$ such that for every $k \in \N$ every digraph~$D$ with $\delta^+(D) \geq g(k)$ contains $k$ vertex-disjoint directed cycles of pairwise distinct lengths.
\end{mainresult}

It is natural to ask whether~\cref{con:lichiardopol} is more generally true in a \emph{weighted} setting, where vertices or edges are assigned positive real weights, and one looks for vertex-disjoint directed cycles of distinct total weights. For edge-weightings, this question was studied already by Henning and Yeo in 2012, who proved that, rather surprisingly, there exist digraphs of arbitrarily high minimum out-degree whose edges can be weighted by positive integers such that all directed cycles have the same weight (cf.~\cite[Theorem~7]{henningyeo}). Hence, the edge-weighted generalisation fails badly. By a reduction to~\cref{main:CyclesDistinctLengths}, we can show that in contrast, the vertex-weighted generalisation of~\cref{con:lichiardopol} \emph{does} hold.

\begin{maincorollary} \label{cor:vertexweights}
    Let $g:\mathbb{N}\rightarrow \mathbb{N}$ be the function from~\cref{main:CyclesDistinctLengths}. Then for every $k\in \mathbb{N}$, every digraph $D$ with $\delta^+(D)\ge g(k)$, and every strictly positive vertex-weighting $w:V(D)\rightarrow \mathbb{R}_+$, there exist $k$ vertex-disjoint directed cycles in $D$ with pairwise distinct total weights according to $w$.
\end{maincorollary}

This has the following consequence, showing that the edge-weighted generalisation holds if the edge-weighting only takes on boundedly many values.

\begin{maincorollary} \label{cor:edgeweights}
    Let $g:\mathbb{N}\rightarrow \mathbb{N}$ be the function from~\cref{main:CyclesDistinctLengths}. Then for every $k,m\in \mathbb{N}$, every digraph $D$ with $\delta^+(D)\ge g(k)\cdot m$, and every edge-weighting $w:E(D)\rightarrow S$ for some set $S\subseteq \mathbb{R}_+$ with $|S|=m$, there exist $k$ vertex-disjoint directed cycles in $D$ with pairwise distinct total weights according to $w$.
\end{maincorollary}

\paragraph{Some remarks on the proof of Theorem 2 and related problems.}
If one does not insist on the obtained cycles in~\cref{con:lichiardopol} being vertex-disjoint, a simple argument (cf.~\cite{bensmail,rapha}) shows that every digraph with minimum out-degree at least~$k$ contains $k$ directed cycles of pairwise distinct lengths. Given this, it is tempting to try and prove \cref{con:lichiardopol} by simply trying to find $k$ vertex-disjoint subdigraphs, each with minimum out-degree at least $k$. One could then greedily pick a directed cycle of length distinct from the previously picked cycles, one per subdigraph, and the disjointness of these subdigraphs would then guarantee that one eventually obtains a collection of $k$ vertex-disjoint cycles of pairwise distinct lengths. 

However, it turns out to be an old, amazingly difficult and well-known open problem, posed independently by Alon and Stiebitz~\cite{stiebitzdigraph,alon,alonsplitting}, to determine whether there exists, for every $s\in \mathbb{N}$, a number~$F(s)$ such that every digraph of minimum out-degree at least~$F(s)$ contains two vertex-disjoint subdigraphs of minimum out-degree at least $s$. See also the recent paper~\cite{noteonsplitting}, reducing this problem to the case $s=2$. In fact, despite being a natural statement, several researchers in the community seem to seriously doubt its correctness. As such, this avenue for attack towards Lichiardopol's conjecture remains wide open and seems well beyond the reach of the current methods in the area. 

Another natural avenue to solve~\cref{con:lichiardopol} would be to show an \emph{Erd\H{o}s-P\'{o}sa-type} statement for disjoint directed cycles of distinct lengths. For example,~\cref{con:lichiardopol} would follow if one could show that there exists a function $h:\mathbb{N}\rightarrow \mathbb{N}$, such that for every $k\in \mathbb{N}$ and every digraph $D$ without $k$ vertex-disjoint directed cycles of distinct lengths, one can delete a set of at most $h(k)$ vertices from $D$ such that the resulting digraph has at most $h(k)$ distinct directed cycle lengths. While this sounds very natural, Gollin et al.~\cite{gollin} recently proved that rather surprisingly, no such function $h$ exists.

Our proof of \cref{main:CyclesDistinctLengths} follows a different route: It builds on and significantly enhances the proof approach followed by the second author in~\cite{rapha}. As such, it combines various ingredients from structural digraph theory. In particular, our proof uses butterfly minors (a by now standard notion of minors in digraphs), directed tangles~\cite{GKKKDirectedToT}, a directed analogue of the \emph{Tangle-Wall Theorem} due to Robertson and Seymour~\cite{GMX}, and a local variant of the \emph{Directed Flat Wall Theorem}~\cite{DirectedFlatWall}. In particular, the usage of directed tangles and the directed tangle-wall theorem are key new ingredients of our proof not used in~\cite{rapha}, which crucially `guide' the proof towards a choice of a `good', well-connected wall, to which previously discovered arguments from~\cite{rapha} can then be applied successfully. (See \cref{subsec:ProofSketch} for a more detailed proof sketch.)

We remark that while we do not compute the function $g$ in~\cref{main:CyclesDistinctLengths} explicitly, one could derive a \emph{computable} function $g$ satisfying the theorem statement by appropriately combining the corresponding results from the literature used in the proof. Since any bound obtained in this way would have to rely on the best known bounds for the \emph{Directed Grid Theorem}~\cite{KKDirectedGridThm}, which are currently a power tower of height~$22$ \cite{HKMM24}, we decided not to optimize the bound for the function~$g$ in~\cref{main:CyclesDistinctLengths}.

\paragraph*{Basic notation and terminology.} We denote by $\defnm{\mathbb{N}}$ the set of natural numbers (excluding $0$). For $k\in \mathbb{N}$, we denote by $\defnm{[k]}=\{1,\ldots,k\}$ the set of the first $k$ integers. All digraphs in this paper are finite and simple, but anti-parallel pairs of edges are allowed. For a digraph $D$, we denote by \defn{$V(D)$} its vertex set and by $\defnm{E(D)}\subseteq \{(u,v)\in V(D)^2|u\neq v\}$ its edge set. For an edge $(u,v)\in E(D)$, we think of it as being directed from $u$ to $v$ and refer to $u$ as the \defn{startpoint} of the edge, and to $v$ as the \defn{endpoint} of the edge.
For a vertex $v\in V(D)$, we denote by \defn{$N_D^+(v)$, $N_D^-(v)$} its out- and in-neighbourhood, respectively, and by \defn{$d_D^+(v)$, $d_D^-(v)$} its out- and in-degree. In each case, we omit the subscript $D$ when it is clear from context. We define the \defn{minimum/maximum out-/in-degree} of~$D$ as $\defnm{\delta^+(D)} := \min_{v\in V(D)}d_D^+(v)$, $\defnm{\Delta^+(D)} := \max_{v\in V(D)}d_D^+(v)$, $\defnm{\delta^-(D)} := \min_{v\in V(D)}d_D^-(v)$, and $\defnm{\Delta^-(D)} := \max_{v\in V(D)}d_D^-(v)$, respectively. For a subset of vertices $X$ of a digraph $D$, we denote by \defn{$D[X]$} the induced subdigraph with vertex set $X$ and we use the notation $\defnm{D-X} := D[V(D)\setminus X]$ for deleting sets of vertices. For $t\in \mathbb{N}$, we denote by \defn{$\bivec{K}_t$} the complete digraph of order~$t$, i.e. the (up to isomorphism) unique digraph on $t$~vertices containing all possible $t(t-1)$ ordered pairs of distinct vertices as edges.

\paragraph*{Organisation.}
In~\cref{sec:Prelims} we start by introducing important terminology and several key definitions from structural digraph theory that will be used throughout the paper. In the following~\cref{sec:aux} we then collect various auxiliary results from the literature that we shall employ in our proofs. This includes a `local' variant of a weakening of the \emph{Directed Flat Wall Theorem} due to Giannopoulou, Kawarabayashi, Kreutzer and Kwon~\cite{DirectedFlatWall}, as well as some results from~\cite{rapha}. In the same section, we also give a more detailed overview/sketch of our proof of \cref{main:CyclesDistinctLengths}. 
In the next section,~\cref{subsec:DirectedTangles}, we then prove two auxiliary results about directed tangles, which are key ingredients in our proof. First, in~\cref{subsec:THEtangle} we prove a key novel result, stating that every digraph of large minimum out-degree which does not contain two vertex-disjoint subdigraphs of still large minimum out-degree, has the property that orienting every low-order vertex-separation according to the `directions of the edges' across the separation, gives rise to a directed tangle.
Second, in~\cref{subsec:CylindWallControlledByTangle} we prove the natural directed analogue of the classical Tangle-Wall Theorem due to Robertson and Seymour~\cite{GMX}, saying that every high-order tangle controls a large wall. We deduce this result as a consequence of known results from the literature on structural digraph theory.
Finally, in \cref{sec:proof} we present the proof of our main result,~\cref{main:CyclesDistinctLengths}, which combines all the aforementioned ingredients and introduces further new ideas. 

After that in~\cref{sec:weighted}, we present the reductions that allow us to deduce the `weighted'~\cref{cor:vertexweights,cor:edgeweights} from~\cref{main:CyclesDistinctLengths}, and finish with an open problem in~\cref{sec:conc}.

\section{Preliminaries} \label{sec:Prelims}

In this section, we collect several basic definitions and terminology from structural digraph theory that will be used repeatedly throughout the paper, and include some basic results and observations. We start by recalling the definition of (vertex-)separations in digraphs.

Let $D$ be a digraph and $A, B \subseteq V(D)$. We say that an edge $e \in E(D)$ \defn{crosses} from $A$ to $B$ if its startpoint is in $A \setminus B$ and its endpoint is in $B \setminus A$. 

\begin{definition}[Directed separation]
    A \defn{(directed) separation} of a digraph $D$ is an unordered pair $\{A,B\}$ of subsets $A,B$ of~$V(D)$ such that $A \cup B = V(D)$ and either no edge of~$D$ crosses from~$A$ to~$B$ or no edge of~$D$ crosses from $B$ to $A$. We will drop the `directed' and refer to $\{A,B\}$ only as `separation'.

    $\{A,B\}$ is \defn{proper} if neither $A$ nor $B$ equals $V(D)$.
    The \defn{order} $|\{A, B\}|$ of $\{A, B\}$ is defined as the size $|A \cap B|$ of its \defn{separator} $A \cap B$. 
    The \defn{orientations} of $\{A,B\}$ are the \defn{oriented separations} $(A,B)$ and $(B,A)$.
\end{definition}

Next, we define \emph{butterfly minors}, which are a key notion of minors for digraphs, and which play a similar role in structural digraph theory as do ordinary graph minors in structural graph theory. This notion is for example used in the \emph{Directed Grid Theorem}~\cite{KKDirectedGridThm} due to Kawarabayashi and Kreutzer and in the \emph{Directed Flat Wall Theorem} due to Giannopoulou et al.~\cite{DirectedFlatWall}, both of which will be stated later. 

\begin{definition}[Butterfly minor]
Let $D$ be a digraph. An edge $e=(u,v) \in E(D)$ is called \defn{contractible} if $d_D^+(u)=1$ or $d_D^-(v)=1$, i.e. if $e$ is the only edge leaving $u$ or the only edge entering $v$. If $e$ is contractible, \defn{contracting~$e$} means transforming $D$ into a new digraph obtained by identifying $u$ and $v$ into a single vertex, and ignoring loops and multiple edges in the same direction created by this process. 

A digraph $D'$ is called a \defn{butterfly minor} of another digraph $D$ if $D'$ is isomorphic to a digraph which can be obtained from $D$ via a finite sequence of edge-deletions, vertex-deletions, and contractions of contractible edges (in arbitrary order).
\end{definition}

It is well-known and not difficult to prove that if $F$ is a digraph of maximum (total) degree at most~$3$ such that $\Delta^+(F), \Delta^-(F) \le 2$, then $F$ is a butterfly minor of~$D$ if and only if $D$ contains a \defn{subdivision} of~$F$, i.e.\ a subdigraph that is (up to isomorphism) obtained from $F$ by replacing its edges by a collection of internally vertex-disjoint directed paths with start- and endpoints corresponding to those of the original edges.
\smallskip

Next, we recall the definition of cylindrical walls, which form the directed analogue of walls in undirected graph structure theory. 

\begin{definition}[(Elementary) cylindrical wall, cf.~\cite{DirectedFlatWall}]\label{def:wall}
The \defn{elementary cylindrical wall of order~$k$} for $k \in \mathbb{N}$ is the planar digraph~\defn{$W_k$} with vertex set $V(W_k)=[2k]\times [2k]$ and in which there is an edge from a vertex $(x,y)$ to another vertex $(x',y')$ if and only if one of the following holds (see~\cref{fig:wall} for an illustration):
\begin{itemize}
    \item $y=y'$ is odd and $x'=x+1$.
    \item $y=y'$ is even and $x'=x-1$. 
    \item $x=x'$ is odd, $y$ is even and $y'\equiv y+1 \text{ (mod }2k)$.
    \item $x=x'$ is even, $y$ is odd and $y'\equiv y+1 \text{ (mod }2k)$.
\end{itemize}

The \defn{rows} of an elementary cylindrical wall~$W$ (of order~$k$) are the induced directed paths $R_n$, for $n \in [2k]$ on vertex set $\{(x,n) : x \in [2k]\}$, and its \defn{columns} are the induced directed cycles $C_n$, for $n \in [k]$, on vertex set $\{(2n-1, y), (2n, y) : y \in [2k]\}$ (see \cref{fig:wall}).

A \defn{cylindrical wall of order $k$} is any digraph isomorphic to a subdivision of $W_k$, i.e.\ a digraph which can be obtained from $W_k$ by replacing its edges by internally disjoint directed paths with the same start- and endpoints. Its rows and columns are the subdivided rows and columns of~$W_k$.

The \defn{perimeter} \defn{$\mathrm{per}(W)$} of a cylindrical wall $W$ is defined to be the set of vertices contained in its first or last column. The \defn{interior} is its complement, i.e. $\defnm{\mathrm{int}(W)} := V(W)\setminus \mathrm{per}(W)$. A \defn{brick} of $W$ is any face of its natural planar or cylindrical embedding which differs from the two faces bounded by $\mathrm{per}(W)$.
\end{definition}

\begin{figure}[ht]
	\centering
    	\begin{tikzpicture}[scale=0.5, decoration={
			markings,
			mark=at position 0.365 with {\arrowreversed{Straight Barb[length=1.8mm,width=1.8mm]}}}]
		\tikzstyle{w}=[circle,draw,fill=black!50,inner sep=0pt,minimum width=3pt]

		\foreach \x in {0, 4, 8}{
			\foreach \y in {1.5,4.5,7.5,10.5}{	
				\draw[postaction={decorate}] (\x+2, \y)-- (\x+2, \y-1.5);	
				
				\draw[postaction={decorate}] (\x, \y)-- (\x+2, \y);	
				\draw[postaction={decorate}] (\x+2, \y)-- (\x+4, \y);	
				
				\draw[postaction={decorate}] (\x+2, \y-1.5)-- (\x, \y-1.5);	
				\draw[postaction={decorate}] (\x+4, \y-1.5)-- (\x+2, \y-1.5);	
				
			}
		}
        \foreach \x in {0, 4, 8}{
			\foreach \y in {1.5,4.5,7.5}{
				\draw[postaction={decorate}] (\x, \y+1.5)-- (\x, \y);		
				
			}
		}
        
		\foreach \x in {0, 12}{
			\foreach \y in {1.5,4.5,7.5,10.5}{	
				\draw[very thick, postaction={decorate}] (\x+2, \y)-- (\x+2, \y-1.5);	
				
				\draw[very thick, postaction={decorate}] (\x, \y)-- (\x+2, \y);	
				
				\draw[very thick, postaction={decorate}] (\x+2, \y-1.5)-- (\x, \y-1.5);	
			}
		}
        \foreach \x in {0, 12}{
			\foreach \y in {1.5,4.5,7.5}{
				
				\draw[very thick, postaction={decorate}] (\x, \y+1.5)-- (\x, \y);		
			}
		}
	
		\foreach \i in {0,2,4,6,8,10,12,14}
		{
			\foreach \j in {0,1.5,3,4.5,6,7.5,9,10.5}
			{
				\draw[fill,black] (\i,\j) circle (.12);
			}
		}

        \foreach \x in {2,4,6,8,10}{
            \draw[fill,lightishGreen] (\x,4.5)  circle (.12);
            \draw[lightishGreen,thick,postaction={decorate}] (\x, 4.5)-- (\x+2,4.5);	
        }
        \foreach \x in {0,12}{
            \draw[lightishGreen,very thick,postaction={decorate}] (\x, 4.5)-- (\x+2,4.5);	
        }
        \foreach \x in {4,6}{
            \foreach \y in {0,1.5,3,4.5,6,7.5,9,10.5}{
            \draw[lightishGreen,fill] (\x,\y) circle (0.12);
            }
        }
        \foreach \x in {0,12,14}{
            \draw[fill,lightishGreen] (\x,4.5)  circle (.12);
        }

        \foreach \y in {0,3,6,9}{
            \draw[lightishGreen,thick,postaction={decorate}] (6,\y)-- (4,\y);	
            \draw[lightishGreen,thick,postaction={decorate}] (6,\y+1.5) -- (6,\y);
            \draw[lightishGreen,thick,postaction={decorate}] (4, \y+1.5) -- (6,\y+1.5);
        }
        \foreach \y in {3,6,9}{
            \draw[lightishGreen,thick,postaction={decorate}] (4,\y) -- (4,\y-1.5);
        }

\begin{scope}[decoration={
        markings, reset marks,
        mark=at position 0.5 with {\arrowreversed{Straight Barb[length=1.8mm,width=1.8mm]}}}]

\draw [very thick,postaction={decorate}] plot [smooth, tension=0.45] coordinates {(12,0) (12.5,-1) (14.5,-0.6) (14.9,5.5) (14.5,10.7) (12.3,11.1) (12,10.5)};
\end{scope}

\begin{scope}[decoration={
        markings, reset marks,
        mark=at position 0.502 with {\arrowreversed{Straight Barb[length=1.8mm,width=1.8mm]}}}]

\draw [postaction={decorate}] plot [smooth, tension=0.4] coordinates {(8,0) (8.8,-1.3) (15.1,-1.1) (15.9,5.5) (15.1,11.2) (8.7,11.4) (8,10.5)};
\end{scope}

\begin{scope}[decoration={
        markings, reset marks,
        mark=at position 0.499 with {\arrowreversed{Straight Barb[length=1.8mm,width=1.8mm]}}}]
\draw [lightishGreen,thick,postaction={decorate}] plot [smooth, tension=0.4] coordinates {(4,0) (5.3,-1.6) (15.4,-1.5) (16.6,5.5) (15.4,11.6) (5.2,11.8) (4,10.5)};
\end{scope}

\begin{scope}[decoration={
        markings, reset marks,
        mark=at position 0.498 with {\arrowreversed{Straight Barb[length=1.8mm,width=1.8mm]}}}]
\draw [very thick,postaction={decorate}] plot [smooth, tension=0.4] coordinates {(0,0) (1.6,-2) (15.7,-1.9) (17.4,5.5) (15.7,12.1) (1.6,12.2) (0,10.5)};
\end{scope}

\draw (-1,0) node {$R_{1}$};
\draw (-1,1.5) node {$R_{2}$};
\draw (-1,3) node {$R_{3}$};
\draw[lightishGreen] (-1,4.5) node {$R_{4}$};
\draw (-1,6) node {$R_{5}$};
\draw (-1,7.5) node {$R_{6}$};
\draw (-1,9) node {$R_{7}$};
\draw (-1,10.5) node {$R_{8}$};

\draw (1,-0.7) node {$C_{1}$};
\draw[lightishGreen] (5,-0.7) node {$C_{2}$};
\draw (9,-0.7) node {$C_{3}$};
\draw (13.2,-0.7) node {$C_{4}$};

\end{tikzpicture}
	\caption{The elementary cylindrical wall of order~$4$. The row $R_4$ and the column $C_2$ are highlighted in green, while the two columns whose vertices form the perimeter are marked with thick lines.}\label{fig:wall}
\end{figure}
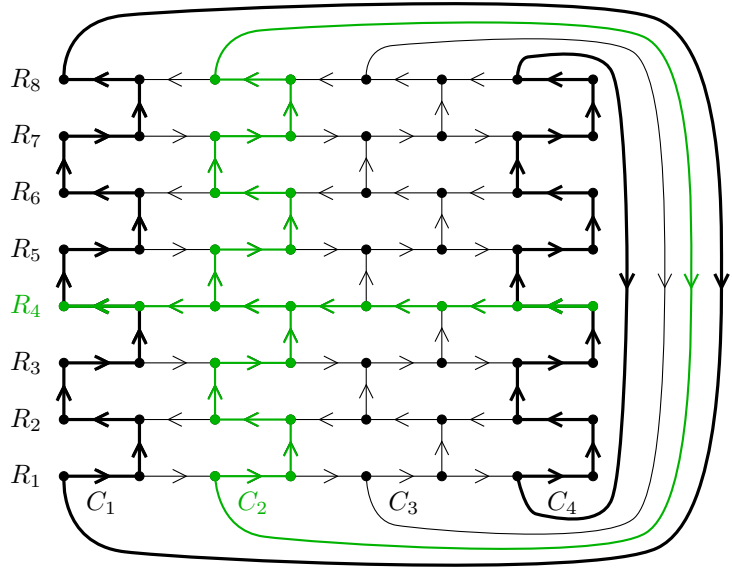

The relevance of cylindrical walls to digraph structure theory is that they form canonical digraphs with large `directed tree-width'. The \emph{directed tree-width} $\mathrm{dtw}(D)$ of a digraph $D$ is a natural number which intuitively measures how `structurally far' $D$ is from an acyclic digraph (the latter correspond exactly to the digraphs $D$ such that $\mathrm{dtw}(D)=0$). It forms a generalisation of the well-known \emph{undirected} tree-width to digraphs. Since it will not be relevant for what follows, we omit a precise definition of directed tree-width, and instead refer the interested reader to the introductory paper of Johnson, Robertson, Seymour and Thomas~\cite{directedtreewidth}. As previously mentioned, cylindrical walls have large directed tree-width (growing to infinity with their order). Conversely, Kawarabayashi and Kreutzer~\cite{KKDirectedGridThm} proved the \emph{Directed Grid Theorem}, stating that a qualitative converse is true as well:

\begin{theorem}[Directed Grid Theorem~\cites{KKDirectedGridThm}{HKMM24}] \label{thm:directedgrid}
There exists a function $d:\mathbb{N}\rightarrow \mathbb{N}$ such that every digraph~$D$ with directed tree-width at least $d(k)$ contains a cylindrical wall of order~$k$ as a subdigraph.
\end{theorem}

\noindent This is a generalisation of the classic \emph{Grid Theorem} of Robertson and Seymour~\cite{MR854606}.

We remark that the original theorem of Kawarabayashi and Kreutzer~\cite{KKDirectedGridThm} deduced the existence of a cylindrical \emph{grid} as a \emph{butterfly minor} instead of a cylindrical wall as a subgraph. However, the qualitative equivalence between these two objects is well-known (see also \cite[Theorem~5.12]{HKMMDirectedGridThm2}, where the precise statement of \cref{thm:directedgrid} in terms of walls and subgraphs can be found).

Before moving on to a discussion of the Directed Flat Wall Theorem, we record the following simple observation about the interaction between the rows of a cylindrical wall and its cylindrical subwalls for later purposes.

\begin{observation}\label{obs:wallintersection}
    Let $W$ and $W'$ be cylindrical walls such that $W'\subseteq W$. Let $X$ be the set of vertices of a row of $W$. Then $V(W')\cap X\neq\emptyset$.
\end{observation}

\begin{proof}
    Observe that $W'$ contains a directed cycle, which is thus also a directed cycle in $W$. Furthermore, note that by definition of a cylindrical wall, every directed cycle in $W$ must meet every row in some vertex. Hence, in particular the directed cycle in $W'$ must intersect $X$, and thus $V(W')\cap X\neq\emptyset$.  
\end{proof} 

The \emph{Directed Flat Wall Theorem} due to Giannopoulou et al.~\cite{DirectedFlatWall} is a qualitative strengthening of the Directed Grid Theorem. It says that in every digraph $D$ of sufficiently large directed tree-width, we can either find a large complete digraph~$\bivec{K}_t$ as a butterfly minor, or we can find a so-called \emph{flat wall} in~$D$, i.e.\ a cylindrical wall~$W$ such that the rest of the digraph interacts in a rather controlled way with~$W$. 

\begin{definition}[Weak flatness,~cf.~\cite{rapha}]
Let $D$ be a digraph, and let $W\subseteq D$ be a cylindrical wall. We say that $W$ is \defn{weakly flat} in~$D$ if for every directed path $P$ in $D$ with endpoints $x,y$ such that $V(P)\cap V(W)=\{x,y\}$ and such that at least one of $x,y$ lies in $\mathrm{int}(W)$, we have that there exists a brick~$B$ of~$W$ with $x,y\in V(B)$.
\end{definition}

In the next section, we will explicitly state and justify a (in some sense weaker, in some sense stronger) version of the Directed Flat Wall Theorem from~\cite{DirectedFlatWall} that uses the above definition.

\section{Proof sketch and auxiliary results} \label{sec:aux}

In this section, we collect several important auxiliary results that will be crucial ingredients in the proof of our main result,~\cref{main:CyclesDistinctLengths}. First, we discuss the aforementioned variant of the Directed Flat Wall Theorem. Then, we give a high-level overview of our proof of~\cref{main:CyclesDistinctLengths}, compare it to an approach previously developed by the second author in~\cite{rapha} and, along the way, state several auxiliary results from that work which will be re-used in our proof of~\cref{main:CyclesDistinctLengths}.

\subsection{A version of the Directed Flat Wall Theorem}
We start by formulating and justifying a variant of the original statement of the Directed Flat Wall Theorem proved by Giannopoulou et al.~\cite{DirectedFlatWall}. As previously mentioned, the original statement of the theorem differs from the one presented here in two ways: (1)~The notion of flatness of a wall used in~\cite{DirectedFlatWall} is strictly stronger than the `weakly flat'-notion we use in the following statement, but at the same time (2)~the original theorem starts from the assumption of high directed tree-width and then guarantees a large complete butterfly minor or \emph{some} flat cylindrical wall. Here, we instead need a `local version', which assumes a large cylindrical wall is given to us to start with, and then guarantees the found flat cylindrical wall to be a \emph{subwall} of that initially fixed, larger, wall. While this second property does not explicitly follow from the standard formulations of the Directed Flat Wall Theorem (\cite[Theorem~2.3]{DirectedFlatWall}), as we shall explain below, the `local version' needed here follows directly from the \emph{proof} of a key lemma (Lemma~4.2) in \cite{DirectedFlatWall}.

\begin{theorem}\label{thm:flatwall}
There exist functions $\rho_{\ref{thm:flatwall}}\colon \mathbb{N}\times\mathbb{N}\to\mathbb{N}$ and $\alpha_{\ref{thm:flatwall}}\colon\N\to\N$ such that, for all $r,t\geq 1$, every digraph $D$, and every cylindrical wall
$W\subseteq D$ of order at least $\rho_{\ref{thm:flatwall}}(r,t)$, one of the following holds:
\begin{enumerate}
  \item\label{itm:FlatWall:CliqueMinor} $D$ contains $\bivec{K}_t$ as a butterfly minor; or
  \item\label{itm:FlatWall:Wall} There exists some $A\subseteq V(D)$ with $|A|\leq\alpha_{\ref{thm:flatwall}}(t)$ and a cylindrical wall $W'\subseteq W-A$
        of order $r$ such that $W'$ is weakly flat in $D-A$.
\end{enumerate}
\end{theorem}

As indicated before, the statement of \cref{thm:flatwall} differs from that of \cite[Lemma~4.2]{DirectedFlatWall} in the following way: \cite[Lemma~4.2]{DirectedFlatWall} starts with a digraph~$D$ of large directed tree-width, and then finds a set $A \subseteq V(D)$ and \emph{some} cylindrical wall~$W'$ in~$D$ such that $W'$ is weakly flat in $D-A$ (if $D$ does not contain $\bivec{K}_t$ as a butterfly minor), while \cref{thm:flatwall} starts with a \emph{given} cylindrical wall~$W$ in~$D$, and asserts that the cylindrical wall~$W'$ can be taken as a cylindrical \emph{subwall} of~$W$. Closely related, but not equivalent statements to the one used here that also start with an arbitrarily given, sufficiently large wall, have appeared explicitly in the literature, see for example~\cite[Theorem~1.8]{MR4764852} and~\cite[Theorem~6.3]{hatzelrecent}.

While not explicitly stated there,~\cref{thm:flatwall} follows immediately from the proof of \cite[Lemma~4.2]{DirectedFlatWall}. Indeed, in the proof of \cite[Lemma~4.2]{DirectedFlatWall}, they first apply the Directed Grid Theorem~\cite{KKDirectedGridThm} to~$D$, which yields \emph{some} (large) cylindrical wall~$W$ in $D$ (as $D$ is assumed to have large directed tree-width). They then construct the cylindrical wall~$W'$ in several steps, where in each step, they either find a suitable collection of `jumps', which can be used to construct a $\bivec{K}_t$ butterfly minor and thus yields outcome \ref{itm:FlatWall:CliqueMinor}, or they find a cylindrical \emph{subwall} $\widetilde{W}$ of~$W$ which after removing a bounded set of vertices (in terms of $t$) has fewer `jumps', and they then proceed with their argument in $\widetilde{W}$. Thus, if they never encounter a $\bivec{K}_t$ butterfly minor, they eventually end up with a cylindrical \emph{subwall}~$W'$ of~$W$ and a set of vertices $A\subseteq V(D)$ of size bounded in terms of $t$, such that every `jump' in $D-A$ across $W'$ must have endpoints in a common brick. Concretely, they guarantee that every directed path in $D-A$ with both endpoints in $W'$ and all internal vertices outside $W'$ such that at least one endpoint is in the interior of $W'$, has both of its endpoints in a common brick of $W'$. Hence, $W'$ is weakly flat in $D-A$.

Therefore, if we run the proof of \cite[Lemma~4.2]{DirectedFlatWall} on input $W$ (where $W$ is the cylindrical wall given by the premises of \cref{thm:flatwall}), then it will return either a $\bivec{K}_t$ butterfly minor in~$D$ (as in~\ref{itm:FlatWall:CliqueMinor}) or a cylindrical \emph{subwall}~$W'$ of~$W$ satisfying~\ref{itm:FlatWall:Wall}.

\subsection{Proof sketch} \label{subsec:ProofSketch}

We now give a high-level sketch of the proof of~\cref{main:CyclesDistinctLengths}, comparing it to the proof strategy of the second author from~\cite{rapha}, highlighting the novelties, and explaining the relevant auxiliary results from that paper on the way. Let $k \in \N$, and let $D$ be some digraph of sufficiently large minimum out-degree $d\gg k$. Our goal is to show that $D$ contains $k$ vertex-disjoint directed cycles of pairwise distinct lengths.

By a result of the second author~\cite[Proposition~6]{rapha}, we may assume that $D$ has large directed tree-width. Therefore, by the Directed Grid Theorem~\cite{KKDirectedGridThm} (see~\cref{thm:directedgrid}), there exists a cylindrical wall~$W$ in~$D$ of large order. Note that since cylindrical walls contain a variety of directed cycles and many disjoint directed cycles, it is tempting to hope that $W$ already will provide us with many vertex-disjoint directed cycles of distinct lengths. However, as shown in~\cite[Remark~3.1]{rapha} and~\cite{gollin}, there exist cylindrical walls of arbitrarily large order in which all directed cycles have the same length. This is why the approach followed in~\cite{rapha} as well as the one we follow here needs to use more information about the digraph $D$ than just the containment of a large-order cylindrical wall. Concretely, this is done through analyzing the interaction of $W$ with the rest of the digraph, and the Directed Flat Wall Theorem turns out to be a useful tool to control this interaction.

By the Directed Flat Wall Theorem~\cite{DirectedFlatWall} (see \cref{thm:flatwall}), either $D$ contains $\bivec{K}_t$, for some $t \gg k$, as a butterfly minor, or there is a small set $A \subseteq V(D)$ and a cylindrical wall $W' \subseteq W$ such that $W'$ is weakly flat in $D-A$. In the first case, we are done by the following result of the second author.

\begin{lemma}[{\cite[Corollary~11]{rapha}}]\label{lem:butterfly}
    Let $k\in \mathbb{N}$, and let $D$ be a digraph containing $\bivec{K}_t$ as a butterfly minor, where $t\ge \frac{k^2+3k}{2}$. Then $D$ contains $k$ vertex-disjoint directed cycles of pairwise distinct lengths.
\end{lemma}

Therefore, we may assume the second case holds. Let $C$ be the strong component of $D-A$ containing~$W'$. If $C$ has no outgoing edge to $D-(A \cup V(C))$, then, since $A$ is small, $C$ has still large minimum out-degree (it only decreased by at most $|A|$). In this case, we are again done by applying the following result of the second author to $W=W'$ and $D=C$ (because $W'$ is weakly flat in $D-A \supseteq C$).  

\begin{lemma}[{\cite[Lemma~3.7]{rapha}}]\label{lem:walltrains}
    Let $k\in \mathbb{N}$, let $D$ be a strongly connected digraph, and let $W\subseteq D$ be a cylindrical wall of order $3k+2$ which is weakly flat in $D$. If $\delta^+(D)\ge 7k-5$, then $D$ contains $k$ vertex-disjoint directed cycles of pairwise distinct lengths.
\end{lemma}

Therefore, we may assume that some edge of $D$ has its startpoint in~$C$ and its endpoint in $D-(A \cup V(C))$.  Let $U$ be the set of all vertices of $D- A$ that can reach $C$ by a directed path in $D-A$ (i.e.\ for every $u \in U$ there exists a directed path from~$u$ to~$C$), and set $X := U \cup A$ and $Y := V(D) \setminus U$. Then $\{X,Y\}$ is a proper separation of $D$ with separator~$A$. Indeed, by definition there can be no edge crossing from $Y$ to $X$. Also, recall that there is some edge starting in $C\subseteq X$ and ending in $V(D)\setminus (A\cup V(C))$. Since $C$ is a strong component of $D-A$, the endpoint of this edge cannot reach $C$ via a directed path in $D-A$. Hence, it cannot lie in $U$ and thus must be contained in $Y$. Thus, this edge crosses from $X$ to $Y$, and in particular the separation is proper. Finally, note that $V(W') \subseteq X\setminus Y$ by definition. 
\smallskip

So to conclude the proof, it would suffice to find a large cylindrical wall~$W$ in~$D$ with the property that for every small-order separation $\{X,Y\}$ of $D$ with some edges crossing from $X$ to $Y$ its strict side $Y \setminus X$ contains a row of~$W$. Indeed, \cref{obs:wallintersection} then ensures that every cylindrical subwall~$W'$ of $W$ intersects $Y \setminus X$.
The above argument would therefore necessarily end before finding a subwall $W'$ and a separation~$\{X,Y\}$ such that $V(W') \subseteq X\setminus Y$ and would return the desired collection of disjoint directed cycles instead.
\smallskip

To find a cylindrical wall $W$ with this property we make use of `directed tangles' (see \cref{def:Tangle}). More concretely, we will define a certain tangle, which will then `guide' us to a `good' wall. 
To define this tangle successfully, we first show that we can solve the problem by induction if in $D$ we can find two vertex-disjoint subdigraphs $D_1, D_2$, both of large minimum out-degree (where by `large', we mean, say, $\delta^+(D)/4$). 
Concretely, suppose that we have proved the conjecture for parameter $k-1$ already, then supposing $\delta^+(D)$ is large enough, we can already find $k-1$ vertex-disjoint directed cycles of distinct lengths in $D_1$. It then suffices to find at least $k$ directed cycles with pairwise distinct lengths within $D_2$ (but not necessarily disjoint), since then one of these cycles in $D_2$ must have a length differing from the $k-1$ lengths of the disjoint cycles found in $D_1$. Adding this cycle then yields $k$ vertex-disjoint directed cycles of distinct lengths in $D$. The last step of this argument is formally handled by the following statement (assuming $\delta^+(D)/4\ge k$).

\begin{lemma}[{\cite[Observation~3.2]{rapha}}] \label{lem:trains}
    Let $k\in \mathbb{N}$. Every digraph $D$ with $\delta^+(D)\ge k$ contains $k$ directed cycles (not necessarily disjoint) with pairwise distinct lengths. 
\end{lemma}

Hence, moving on in the proof, we can assume that there are no two vertex-disjoint subdigraphs of large minimum out-degree.

We then prove a key lemma, \cref{lem:DefaultOrientationIsATangle}. It is one of the main technical innovations of this paper and states that under the above assumption, orienting each small-order separation $\{X,Y\}$ of~$D$ towards its side that contains the endpoints of all the crossing edges of~$D$ yields a tangle~$\tau$. (See the explanation in the beginning of \cref{subsec:THEtangle} for more details.) 

We then apply a `tangle-version' of the Directed Grid Theorem (see \cref{thm:CylindricalWallControlledByTangle}) to~$\tau$, which yields a large cylindrical wall~$W$ in~$D$ which is `controlled' by~$\tau$. Roughly speaking, `controlled' here means that for every small order separation, the tangle $\tau$ orients it towards the (unique) side of the separation containing a `substantial' part of $W$. In particular, since $\tau$ orients any small-order separation $\{X,Y\}$ of~$D$ with edges crossing from~$X$ to~$Y$ towards~$Y$, it forces the separation $\{X,Y\}$ to contain a row of~$W$ in its strict side $Y\setminus X$ (see \cref{thm:CylindricalWallControlledByTangle}). Thus, the wall~$W$ has the aforementioned desired property. Starting the initially described argument from this `good' wall then finishes the proof of the theorem.
\smallskip

Summarizing, the key novelty of our proof compared to the ideas from~\cite{rapha} lies in the idea of finding such a `good choice' of a wall rather than selecting it arbitrarily. For guiding the choice of this wall, directed tangles turn out to be the perfect complementary tool. Secondly, the idea of inducting over $k$ when finding disjoint subdigraphs of high out-degree is novel, too.

\section{Directed tangles} \label{subsec:DirectedTangles}

In this section, we prove two key auxiliary results about directed tangles. The latter form a way of `consistently' orienting all low-order directed separations in a digraph. Directed tangles were first introduced by Giannopoulou, Kawarabayashi, Kreutzer and Kwon~\cite{GKKKDirectedToT}, who generalised some of the well-developed theory of \emph{undirected} tangles to digraphs. The latter were first introduced by Robertson and Seymour in their famous Graph-Minors Series~\cite{GMX}, and have received much further attention since, see Diestel's book~\cite{tangles} for an overview of the topic and its applications.

We start by recalling the definition of directed tangles from~\cite{GKKKDirectedToT}.

\begin{definition}[Directed tangle, cf.\ \cite{GKKKDirectedToT}] \label{def:Tangle}
    Let $D$ be a digraph and $k \in \N$. A set $\tau$ of oriented separations of~$D$ of order $<k$ is called a \defn{$k$-tangle} in~$D$, if 
    \begin{enumerate}[label=\rm{(T\arabic*)}]
        \item \label{itm:Def:Tangle:1} for all separations $\{A,B\}$ of $D$ of order $< k$, we have $(A,B) \in \tau$ or $(B,A) \in \tau$, and
        \item \label{itm:Def:Tangle:2} if $(A_1, B_1), (A_2, B_2), (A_3, B_3) \in \tau$, then $A_1 \cup A_2 \cup A_3 \neq V(D)$.
    \end{enumerate}
\end{definition}

\noindent With this definition of tangle, we follow the definition of tangles in digraphs of \cite{GKKKDirectedToT} and (slightly) deviate from the standard definition of tangles in undirected graphs (as e.g.\ in \cite{Bibel}).

By definition, every $k$-tangle contains \emph{precisely} one of $(A,B)$ and $(B,A)$ for every separation $\{A,B\}$ of~$D$ of order $< k$: \ref{itm:Def:Tangle:1} implies that it contains at least one, and \ref{itm:Def:Tangle:2} implies that it contains at most one (as $A \cup B = V(D)$ for every separation $\{A,B\}$ of~$D$, this follows by taking $(A_1,B_1) := (A_2, B_2) := (A,B)$ and $(A_3,B_3) := (B,A)$ in~\ref{itm:Def:Tangle:2}).

\subsection{The `default orientation' tangle} \label{subsec:THEtangle}

In undirected graphs, for a separation $\{A,B\}$, there is no edge with one end in $A\setminus B$ and one in $B \setminus A$. 
In directed graphs~$D$ however, a separation $\{A,B\}$ is allowed to have such edges as long as they all agree to cross from $A$ to $B$ (or the other way around).
In the case that we will be interested in here, i.e.\ when $D$ has minimum out-degree at least $\delta^+(D)\ge d$ but does not contain two vertex-disjoint subdigraphs both of minimum out-degree at least~$d/4$ (cf.\ paragraph before \cref{lem:trains}), all proper separations of order $<d/4$ will be crossed by at least one edge of~$D$. Indeed, if there was a proper separation $\{A,B\}$ of~$D$ of order $<d/4$ that is not crossed by any edge of~$D$, then $D[A\setminus B]$ and $D[B\setminus A]$ are both subdigraphs of minimum out-degree $\delta^+(D) - |A \cap B| \geq d - d/4 = 3d/4$.  
Note that $\{A,B\}$ being proper ensures that both $A \setminus B$ and $B \setminus A$ are non-empty.
Therefore, we obtain a set~$\tau$ satisfying \ref{itm:Def:Tangle:1} for $k=d/4$ by letting $(A,B) \in \tau$ whenever $\{A, B\}$ is a separation of $D$ of order $<d/4$ that is crossed by an edge from $A$ to $B$. Intuitively, one can think of $(A,B)$ as the `default' orientation of $\{A, B\}$ induced by the crossing edges. (Formally, to satisfy \ref{itm:Def:Tangle:1}, we also need to include $(A, V(D))$ or $(V(D), A)$ in $\tau$ for all non-proper separations $\{A,V(D)\}$ of $D$ of order $<d/4$. To have any chance that $\tau$ is a tangle, we need to put $(A, V(D)) \in \tau$).  
The following lemma asserts that the set~$\tau$ of oriented separations defined in this way is indeed a tangle.
\begin{lemma} \label{lem:DefaultOrientationIsATangle}
    Let $d \in \N$ with $d \geq 4$, and let $D$ be a digraph with $\delta^+(D)\ge d$. Let $\tau$ be the set comprising 
    \begin{itemize}
        \item $(A,B)$ for all proper separations $\{A, B\}$ of $D$ of order $<d/4$ such that no edge crosses from~$B$ to~$A$, and
        \item $(A, V(D))$ for all $A \subseteq V(D)$ of size $< d/4$.
    \end{itemize} 
    If $D$ does not contain two vertex-disjoint subdigraphs of minimum out-degree $\geq d/4$, then $\tau$ is a $(d/4)$-tangle.
\end{lemma}

\begin{proof}
    By definition, $\tau$ satisfies \ref{itm:Def:Tangle:1}.
    To verify \ref{itm:Def:Tangle:2}, let us suppose for a contradiction that there are separations $(A_1, B_1), (A_2, B_2), (A_3, B_3) \in \tau$ such that $A_1 \cup A_2 \cup A_3 = V(D)$. Let
    \[
    \defnm{S} := (A_1 \cap B_1) \cup (A_2 \cap B_2) \cup (A_3 \cap B_3),
    \]
    and let $\defnm{A'_i} := A_i \setminus S$ and $\defnm{B'_i} := B_i \setminus S$ for all $i \in [3]$. By assumption, $\defnm{B'} := B'_1 \cap B'_2 \cap B'_3 = \emptyset$.

    \begin{figure}[ht]
        \centering
        \includegraphics[width=0.6\linewidth]{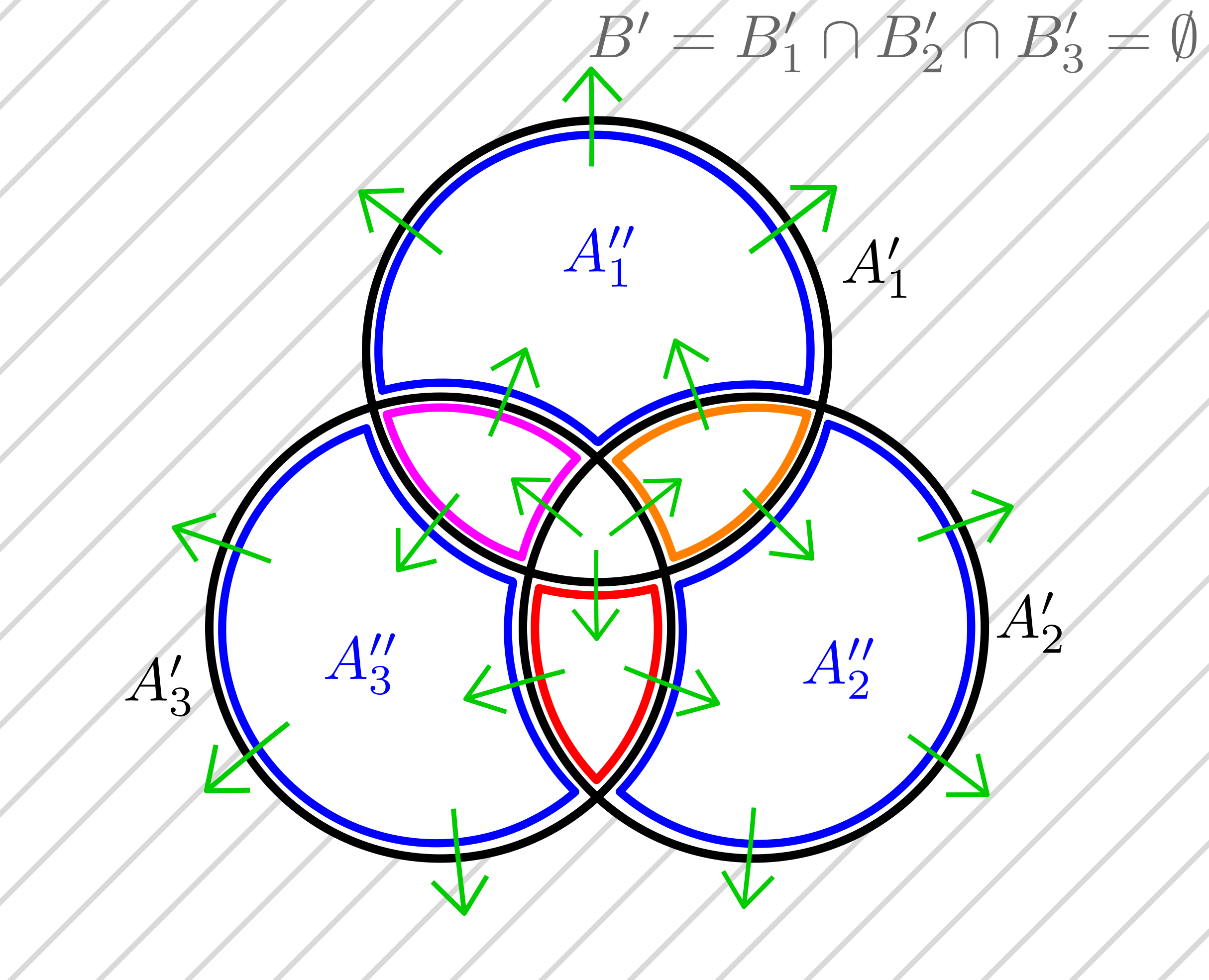}
        \caption{A Venn diagram of the sets $A'_i$, for $i \in [3]$, from the proof of \cref{lem:DefaultOrientationIsATangle}.  In $V(D) \setminus S$, the complement of $A'_1 \cup A'_2 \cup A'_3$ is $B' = B'_1 \cap B'_2 \cap B'_3$. The green arrows indicate in which direction edges of~$D$  are allowed to cross.
        By \cref{claim:DefaultOrientTangle:2}, the sets $A''_i$ (indicated in blue) are empty. By \eqref{eq:DefaultOrientationTangle:2}, $B'_1$ is contained in $\big(A'_2 \cap A'_3) \setminus A'_1\big)$ (in red), $B'_2$ is contained in $\big(A'_1 \cap A'_3) \setminus A'_2\big)$ (in pink), and $B'_3$ is contained in $\big(A'_1 \cap A'_2) \setminus A'_3\big)$ (in orange). Hence, $B'_1, B'_2, B'_3$ are disjoint. By \cref{claim:DefaultOrientTangle:1}, this contradicts the assumption on~$D$.}
        \label{fig:DefaultOrientationTangle}
    \end{figure}
    
    \begin{claim} \label{claim:DefaultOrientTangle:1}
        For all $i \in [3]$, the subdigraph $D[B'_i]$ of~$D$ is non-empty and satisfies $\delta^+(D[B'_i]) \geq d/4$.
    \end{claim}

    \begin{claimproof}
        We first show that $B'_i$ is non-empty.
        Since $\delta^+(D) \geq d$, we have $|V(D)| > d$, which implies that the separation $\{V(D), V(D)\}$ has order~$>d \geq d/4$, and thus $(V(D), V(D))$ is not contained in~$\tau$.
        By the definition of~$\tau$, it follows that $A_i \neq V(D)$, which in turn implies that $B_i \setminus A_i \neq \emptyset$. 
        
        Let $v \in B_i \setminus A_i$. Again by the definition of~$\tau$, no edge crosses from~$B_i$ to~$A_i$, which implies that $N^+_D(v) \subseteq B_i$. Since $\delta^+(D) \geq d$, it follows that $|B_i| \geq d$. As $|S| < 3d/4$, we have $|B'_i| = |B_i \setminus S| \geq |B_i| - |S| \geq d - 3d/4 \geq 1$, and therefore $B'_i$ is not empty.
        \medskip

        By the definition of $\tau$, no edge of~$D$ crosses from $B'_i \subseteq B_i$ to $A'_i \subseteq A_i$. Hence, for every $v \in B'_i$, its out-neighbourhood $N^+_D(v)$ is contained in $B'_i \cup S$. As $|S| < 3d/4$ and $|N^+_D(v)| \geq \delta^+(D) \geq d$, it follows that $|N^+_D(v) \cap B'_i| \geq d/4$, and thus $d^+_{D[B'_i]}(v) \geq d/4$.
        This implies that $\delta^+(D[B'_i]) \geq d/4$. 
    \end{claimproof}

    Let $\defnm{A''_1} := A'_1 \setminus (A'_2 \cup A'_3)$, and analogously $\defnm{A''_2} := A'_2 \setminus (A'_1 \cup A'_3)$ and $\defnm{A''_3} := A'_3 \setminus (A'_1 \cup A'_2)$ (for an illustration of the $A''_i$'s see \cref{fig:DefaultOrientationTangle}). 

    \begin{claim} \label{claim:DefaultOrientTangle:2}
        $A''_i = \emptyset$ for all $i \in [3]$.
    \end{claim}

    \begin{claimproof}
        Without loss of generality let $i = 1$. 
        We show that $d^+_{D[A''_1]}(v)\geq d/4$ for every $v \in A''_1$. Let us first explain how this concludes the proof. If $A''_1 \neq \emptyset$, it follows that $D[A''_1]$ is non-empty and satisfies $\delta^+(D[A''_1]) \geq d/4$. Since $A''_1 \subseteq A'_1 = A_1 \setminus S \subseteq A_1 \setminus B_1$ and $B'_1 \subseteq B_1$, the two subdigraphs $D[A''_1]$ and $D[B'_1]$ are vertex-disjoint. As $D[B'_1]$ is also non-empty and satisfies $\delta^+(D[B'_1]) \geq d/4$ by \cref{claim:DefaultOrientTangle:1}, this then contradicts the assumptions of~$D$.
        \medskip

        Therefore, it suffices to show that $d_{D[A''_1]}^+(v) \geq d/4$ for every $v \in A''_1$. By the definition of~$\tau$ and because
        \[
            A''_1 = A'_1 \setminus (A'_2 \cup A'_3) = A_1 \setminus (S \cup A'_2 \cup A'_3) \subseteq A_1 \setminus (A_2 \cup A_3) \subseteq V(D) \setminus (A_2 \cup A_3) \subseteq (B_2\setminus A_2) \cap (B_3\setminus A_3),
        \]
        no edge in $D$ starts in $A''_1$ and ends in $A'_2 \subseteq A_2$ or in $A'_3 \subseteq A_3$.
        
        As $V(D) = A'_1\; \dot\cup\; S\; \dot\cup\; B'_1$ and $A''_1 = A'_1 \setminus (A'_2 \cup A'_3)$, we have 
        \[
            V(D) \setminus (A''_1 \cup S) = (A'_1 \cup B'_1) \setminus A''_1 = (A'_1 \cup B'_1) \setminus \big(A'_1 \setminus (A'_2 \cup A'_3)\big) \subseteq B'_1 \cup A'_2 \cup A'_3.
        \]
        Combining this with 
        \begin{equation} \label{eq:DefaultOrientTangle:1}
            \begin{aligned}
                B'_1 &\subseteq (B'_1 \cap B'_2) \cup (B'_1 \cap A'_2) \subseteq (B'_1 \cap B'_2 \cap B'_3) \cup (B'_1 \cap B'_2 \cap A'_3) \cup (B'_1 \cap A'_2)\\
                &\subseteq B' \cup A'_3 \cup A'_2 \overset{B' = \emptyset}{=} A'_2 \cup A'_3
            \end{aligned}
        \end{equation}
        yields
        \[
        V(D) \setminus (A''_1 \cup S) \subseteq A'_2 \cup A'_3
        \]
        
        (see also \cref{fig:DefaultOrientationTangle}). Together with our earlier observation that no edge of~$D$ starts in $A''_1$ and ends in $A'_2 \cup A'_3$, this implies that $N^+_D(v) \subseteq A''_1 \cup S$ for every $v \in A''_1$. Since $\delta^+(D) \geq d$ and $|S| < 3d/4$, it follows that $d^+_{D[A''_1]}(v) \geq d/4$.
        As explained previously, this concludes the proof of the claim.
    \end{claimproof}

    By definition, $A'_1 \cap B'_1 = \emptyset$, and therefore $B'_1 = B'_1 \setminus A'_1$. It follows from \eqref{eq:DefaultOrientTangle:1} and \cref{claim:DefaultOrientTangle:2} that
    \begin{equation} \label{eq:DefaultOrientationTangle:2}
        \begin{aligned}
            B'_1 &= B'_1 \setminus A'_1 \overset{\eqref{eq:DefaultOrientTangle:1}}{\subseteq} (A'_2 \cup A'_3) \setminus A'_1 = \big((A'_2 \cap A'_3) \cup (A'_2 \setminus A'_3) \cup (A'_3 \setminus A'_2)\big) \setminus A'_1\\ 
            &= \big((A'_2 \cap A'_3) \setminus A'_1\big) \cup A''_2 \cup A''_3 \overset{\rm{\cref{claim:DefaultOrientTangle:2}}}{=} \big((A'_2 \cap A'_3) \setminus A'_1\big).
        \end{aligned}
    \end{equation}
    Similarly, $B'_2 \subseteq \big((A'_1 \cap A'_3) \setminus A'_2\big)$ and $B'_3 \subseteq \big((A'_1 \cap A'_2) \setminus A'_3\big)$. Hence, $B'_1, B'_2, B'_3$ are pairwise disjoint (cf.\ \cref{fig:DefaultOrientationTangle}). By \cref{claim:DefaultOrientTangle:1}, this contradicts that $D$ does not contain two vertex-disjoint subdigraphs of minimum out-degree $\geq d/4$. This contradiction shows that our initial assumption was false, $\tau$ indeed satisfies also~\cref{itm:Def:Tangle:2} and is thus a tangle, and we may conclude the proof.
\end{proof}

\subsection{Cylindrical wall controlled by a tangle} \label{subsec:CylindWallControlledByTangle}

For undirected graphs~$G$, it is known that for every $k \in \N$ and every tangle~$\tau$ in~$G$ of order $\gg k$, there is some wall of order~$k$ in~$G$ that essentially `induces' the orientations of small-order separations in~$\tau$, in the sense that it is mostly contained on the $B$-side of every separation $(A,B) \in \tau$ of order less than~$k$ (see e.g.\ \cite[(7.5)]{GMX} or \cite[Lemma~14.6]{KTW20}). In this section, we prove the following analogue of this result for directed graphs, by combining results from previous work in the literature.

\begin{theorem}  \label{thm:CylindricalWallControlledByTangle}
    There exists a function $f_{\ref{thm:CylindricalWallControlledByTangle}}: \N \to \N$ such that the following holds for all digraphs~$D$ and $k \in \N$. 
	Let $\tau$ be an $f_{\ref{thm:CylindricalWallControlledByTangle}}(k)$-tangle in~$D$. Then there is a cylindrical wall $W \subseteq D$ of order~$3k-2$ such that for all separations $\{A,B\}$ of~$D$ of order~$<k$
    \[
	(A,B) \in \tau\quad \Leftrightarrow\quad \text{ there exist a row and a column of } W \text{ that are contained in } D[B\setminus A].
	\]
\end{theorem}

In particular, we obtain the following more `tangle-like' formulation of \cref{thm:CylindricalWallControlledByTangle}, which extends \cite[Lemma~14.6]{KTW20} verbatim to digraphs.

\begin{corollary}
    There exists a function $f: \N \to \N$ such that for every $k \in \N$, every digraph~$D$, and every tangle~$\tau$ of~$D$ of order~$\geq f(k)$ there is a cylindrical wall $W \subseteq D$ of order $3k-2$ such that $\tau_W \subseteq \tau$. \qed
\end{corollary}

Here, $\tau_W$ is the set of oriented separations $(A, B)$ of~$D$ of order less than~$k$ such that $D[B \setminus A]$ contains both a column and a row of~$W$ (that such a side exists for every separation $\{A,B\}$ of $D$ of order~$<k$ is straightforward to check).
\medskip

We need two ingredients for the proof of \cref{thm:CylindricalWallControlledByTangle}. 
The first ingredient is the Directed Grid Theorem of Kawarabayashi and Kreutzer~(\cref{thm:directedgrid}), which we will use as a black box. More precisely, we use a stronger version of the Directed Grid Theorem that gives a directed grid with respect to some given well-linked set.

\begin{definition}[Well-linked set]
    A set~$X$ of vertices of a digraph~$D$ is \defn{well-linked} if for all subsets $Y, Z$ of $X$ with $|Y| = |Z|$ there are $|Y|$ pairwise vertex-disjoint directed paths starting in~$Y$ and ending in~$Z$ in~$D$ and $|Y|$ pairwise vertex-disjoint directed paths starting in~$Z$ and ending in~$Y$ in~$D$.
\end{definition}

The following result is due to Kawarabayashi and Kreutzer (see also \cite[Theorem~2.4]{KKKXHalfIntEPDirectedOddCycles}).

\begin{theorem}[{\cite[Theorem~7.1]{KKDirectedGridThmArXiv}}] \label{thm:GridTheorem:WellLinkedSet}
    There exists a function $f_{\ref{thm:GridTheorem:WellLinkedSet}}: \N \to \N$ such that the following holds for all digraphs~$D$ and $k \in \N$. Let $X \subseteq V(D)$ be a well-linked set of size $\geq f_{\ref{thm:GridTheorem:WellLinkedSet}}(k)$. Then there is a cylindrical wall~$W$ of order~$k$ in~$D$ such that for every set~$U$ of $k$ vertices of~$W$ of total degree at least three in $W$ (i.e.\ of either out-degree~$\geq 2$ or in-degree~$\geq 2$), there are $k$ vertex-disjoint paths from~$X$ to~$U$ and from~$U$ to~$X$.
\end{theorem}

The second ingredient to the proof of \cref{thm:CylindricalWallControlledByTangle} are `brambles'.

\begin{definition}[Bramble]
    A \defn{bramble} in a digraph~$D$ is a set $\cB$ of strongly connected subsets\footnote{A set $B$ of vertices of a digraph~$D$ is \defn{strongly connected} if $D[B]$ is strongly connected.} $B \subseteq V(D)$ such that $B \cap B' \neq \emptyset$ for all $B \neq B'$.

    A \defn{cover} of~$\cB$ is a set~$X$ of vertices of~$D$ such that $X \cap B \neq \emptyset$ for all $B \in \cB$. The \defn{order} of~$\cB$ is the minimum size of a cover of~$\cB$.
\end{definition}

Again, with this, we follow the definition of brambles in digraphs of \cite{GKKKDirectedToT} and (slightly) deviate from the standard definition of a bramble in undirected graphs (as e.g.\ in \cite{Bibel}).
\medskip

It is well-known that the existence of high-order tangles and brambles are qualitatively equivalent (and both are equivalent to having high directed tree-width).
The following lemma due to Giannopoulou et al.~\cite{GKKKDirectedToT} constructs a bramble from a tangle.

\begin{lemma}[{\cite[Lemmas 6.9 \& 6.10]{GKKKDirectedToT}\protect\footnote{That $\cB_\tau$ is a bramble of order $\geq k$ is shown in the proof of \cite[Lemma~6.10]{GKKKDirectedToT}, as can be seen from its  first two sentences.}}] \label{lem:BrambleFromTangle}
     Let $k \in \N$, and let $\tau$ be a $(2k-1)$-tangle in a digraph~$D$. For every set $S$ of less than~$k$ vertices of~$D$ there exists exactly one strong component $C_S$ of $D-S$ such that $V(C_S) \subseteq B$ for all $(A,B) \in \tau$ with $A \cap B = S$. 
     Moreover, $\cB_\tau := \{V(C_S) : S \subseteq V(D) \text{ with } |S| < k\}$ is a bramble of order~$\geq k$. 
\end{lemma}

The following lemma shows that any sizewise minimal cover of $\cB_\tau$ is closely related to~$\tau$, in the sense that it is mostly contained in the $B$-side of every small-order separation $(A,B) \in \tau$.

\begin{lemma} \label{lem:CoverOfBrambleLiesOnBigSide}
    Let $k \in \N$, and let $\tau$ be a $(2k-1)$-tangle in a digraph~$D$. Further, let $\cB_\tau$ be the bramble defined in \cref{lem:BrambleFromTangle}, and let $X$ be any sizewise minimal cover of $\cB_\tau$. Then for every separation $(A,B) \in \tau$ it holds that $|A \cap X| \leq |A \cap B|$.
\end{lemma}

\begin{proof}
    By the definition of $\cB_\tau$ in \cref{lem:BrambleFromTangle} (applied to $S := A \cap B$), there is some $C \in \cB_\tau$ such that $C \subseteq B\setminus A$. Since every $C' \in \cB_\tau$ meets $C$ by the definition of a bramble, every $C' \in \cB_\tau$ meets $B\setminus A \supseteq C$. Thus, if some $C' \in \cB_\tau$ is not contained in $B\setminus A$, it must meet $A \cap B$. This implies that 
    \[
    X' := \big(X \cap (B\setminus A)\big) \cup (A \cap B) = (X \setminus A) \cup (A \cap B)
    \]
    is a cover of $\cB_\tau$. As $X$ is a minimal cover of $\cB_\tau$, it follows that $|X'| \geq |X|$, and hence $|A \cap B| \geq |X \cap A|$.
\end{proof}

It is known that every sizewise minimal cover of a bramble in a digraph is well-linked (see e.g.\ Giannopoulou et al.~\cite[page 27]{GKKKDirectedToT}, who attribute the result to Reed).

\begin{lemma} \label{lem:BrambleCoverIsWellLinked}
    Every sizewise minimal cover of a bramble in a digraph is well-linked.
\end{lemma}

\arXivOrNot{For the sake of completeness, we include a proof here.
{\begin{proof}
    Let $\cB$ be a bramble of a digraph~$D$ of order $k \in \N$, and let $X$ be a cover of $\cB$ of size~$k$. 

    Suppose for a contradiction that $X$ is not well-linked, i.e.\ there exist sets $Y, Z \subseteq X$ with $|Y| = |Z| =: \ell$ such that there are no $\ell$ pairwise vertex-disjoint paths from~$Z$ to~$Y$. By Menger's Theorem in digraphs, there exists a separation $\{A,B\}$ of $D$ of order $<\ell$ such that $Y \subseteq A$ and $Z \subseteq B$ and no edge crosses from $B$ to $A$.

    The set $(X \setminus Y) \cup (A \cap B)$ has size $\leq |X| - |Y| + |A \cap B| \leq k - \ell + (\ell-1) = k-1 < |X|$. By the minimality of $X$, there is a set $C \in \cB$ that avoids $(X\setminus Y) \cup (A \cap B)$. Since $X$ meets $C$ but $(X \setminus Y) \cup (A \cap B)$ avoids~$C$, the set $(C \cap Y) \setminus (A \cap B)$ is non-empty. In particular, $C$ meets $Y\setminus (A\cap B) \subseteq A \setminus B$. As $C$ is strongly connected and avoids $A \cap B$, it follows that $C \subseteq A\setminus B$.

    A symmetric argument yields that there is a set $C' \in \cB$ such that $C' \subseteq B\setminus A$. This contradicts the definition of a bramble as $A\setminus B$ and $B\setminus A$ are disjoint, and thus $C$ and $C'$ are disjoint.
\end{proof}}%
}{For the sake of completeness, we included a proof in the arXiv version of this paper.}

We now prove \cref{thm:CylindricalWallControlledByTangle} by combining \cref{thm:GridTheorem:WellLinkedSet} with \cref{lem:BrambleFromTangle,lem:BrambleCoverIsWellLinked,lem:CoverOfBrambleLiesOnBigSide}.

\begin{proof}[Proof of \cref{thm:CylindricalWallControlledByTangle}]
    We prove the assertion with the function $f:\N\to\N$ defined by 
    \[
    \defnm{f}(k) := 2\cdot f_{\ref{thm:GridTheorem:WellLinkedSet}}(3k-1)-1,
    \]
    where $f_{\ref{thm:GridTheorem:WellLinkedSet}}$ is the function from \cref{thm:GridTheorem:WellLinkedSet}.
    
    Let $\defnm{\cB} := \cB_\tau$ be the bramble obtained from applying \cref{lem:BrambleFromTangle} with parameter $k':=f_{\ref{thm:GridTheorem:WellLinkedSet}}(3k-1)$ to~$\tau$, and let $X$ be a sizewise minimal cover of $\cB$. Note that $\cB$ has order at least $k'= f_{\ref{thm:GridTheorem:WellLinkedSet}}(3k-1)$, and thus $|X| \geq f_{\ref{thm:GridTheorem:WellLinkedSet}}(3k-1)$. 
    By \cref{lem:BrambleCoverIsWellLinked}, $X$ is well-linked. Therefore, by \cref{thm:GridTheorem:WellLinkedSet}, there exists a cylindrical wall~$W'$ in~$D$ of order~$3k-1$ such that for every set~$U$ of $3k-1$ vertices of~$W'$ of total degree at least~$3$ (in~$W'$), there are $3k-1$ vertex-disjoint directed paths from~$X$ to $U$ and from $U$ to $X$ in $D$.
    
    Let \defn{$W$} be the subwall of~$W'$ of order~$3k-2$ that is obtained from $W'$ by removing all the vertices of the first column of $W'$ (and afterwards removing all vertices of total degree~$1$ until there are no more vertices of total degree~$1$) and removing all vertices and edges of the last and penultimate row of~$W'$ that are not contained in a column of~$W'$.
    (Note that the vertices of $W'$ that are contained in a row of $W'$ but not in any column are precisely those vertices of degree~$2$ in $W'$ that subdivide an edge of the elementary wall that is not contained in any column.)
    We claim that $W$ has the property claimed in the theorem statement.
    To verify this, consider any separation $(A,B) \in \tau$ of order less than~$k$. We first show the following claim.

    \begin{claim} \label{claim:GridThm:ULiesInB}
        $|U \cap (B \setminus A)| \geq k$ for every set $U$ of at least $3k-2$ vertices of~$W$ that have total degree~$3$ in~$W'$.
    \end{claim}

    \begin{claimproof}
        Let us first assume that no edge crosses from $B$ to $A$.
        By the choice of $W'$ via \cref{thm:GridTheorem:WellLinkedSet}, there is a collection $\cP$ of $3k-2$ vertex-disjoint directed paths from~$X$ to~$U$. 
        By \cref{lem:CoverOfBrambleLiesOnBigSide} and because $(A,B) \in \tau$, at most $|A \cap B| \leq k-1$ vertices of~$X$ lie in~$A$, which implies that at least $(3k-2)-(k-1) = 2k-1$ of the paths in $\cP$ start in $X \cap (B\setminus A)$. As no edge crosses from $B$ to $A$, any path in~$\cP$ that starts in~$B$ and ends in~$A$ has to meet $A \cap B$. 
        Since $|A \cap B| \leq k-1$, it follows that at least $(2k-1)-(k-1) = k$ paths in $\cP$ end in $B \setminus A$, and thus $|U \cap (B\setminus A)| \geq k$. 

        Analogously, if no edge crosses from~$A$ to~$B$, then we can use the $3k-2$ vertex-disjoint paths from~$U$ to~$X$ (whose existence is again guaranteed by \cref{thm:GridTheorem:WellLinkedSet}) to conclude that $|U \cap (B \setminus A)| \geq k$. 
    \end{claimproof}
    
    We first show that at least one column of $W$ is contained in $D[B\setminus A]$.
    For this, let $U \subseteq V(W)$ be a set consisting of one vertex of each column of~$W$ such that every vertex in $U$ has degree~$3$ (in~$W'$). Note that $W$ has $3k-2$ columns, and thus $|U| = 3k-2$. By \cref{claim:GridThm:ULiesInB}, we have $|U \cap (B \setminus A)| \geq k$. By the choice of $U$, it follows that there are at least~$k$~columns~$C_{i_1}, \dots, C_{i_k}$ of~$W$ that meet $B\setminus A$.
    Observe that the columns~$C_{i_j}$ of~$W$ are directed cycles and pairwise vertex-disjoint. Since $|A \cap B| \leq k-1$ and every column~$C_{i_j}$ meets $B\setminus A$, it follows that at least one column $C_{i_j}$ avoids $A$ and is thus contained in $D[B\setminus A]$ (because $\{A,B\}$ is a directed separation and each column $C_{i_j}$ is strongly connected).
    \medskip

    To see that at least one row of $W$ is contained in $D[B\setminus A]$, pick a set $U$ consisting of one vertex of every second row of~$W$ as follows. If no edge crosses from $A$ to $B$, then we choose the first vertices of the second row, fourth row, and so on. Otherwise, if no edge crosses from $B$ to $A$, then we choose the first vertices of the first row, third row, and so on.
    Note that $W$ has $6k-4$ rows, and thus $|U| = 3k-2$. Moreover, note that by the definition of~$W$, all vertices in~$U$ have degree~$3$ in~$W'$ (even though they have only degree~$2$ in~$W$). 
    Hence, we may apply \cref{claim:GridThm:ULiesInB} to~$U$, which implies that $|U \cap (B\setminus A)| \geq k$.
    
    Therefore, by the choice of~$U$, there are at least~$k$~rows $R_{i_1}, \dots, R_{i_k}$ of~$W$ that meet $B\setminus A$ (in their vertex contained in~$U$). 
    Note that we chose $U$ so that each row~$R_{i_j}$, for $j \in [k]$, is directed \emph{towards} its vertex in~$U$ if no edge crosses from~$A$ to~$B$, and \emph{away} from its vertex in~$U$ if no edge crosses from~$B$ to~$A$. Therefore, if some row~$R_{i_j}$, for $j \in [k]$ meets $A$, then it also meets $A \cap B$. Since $|A \cap B| \leq k-1$ and the rows~$R_{i_j}$ are pairwise vertex-disjoint, it follows that at least one row~$R_{i_j}$ is contained in $D[B\setminus A]$, as desired. Having shown that $D[B\setminus A]$ contains a row and a column of $W$, we may conclude the proof.
\end{proof}

\section{Putting it all together (Proof of \texorpdfstring{\cref{main:CyclesDistinctLengths}}{Theorem 2})} \label{sec:proof}

Having assembled all the necessary ingredients, we are now finally ready for the proof of our main result.

\begin{proof}[Proof of~\cref{main:CyclesDistinctLengths}]
In the remainder of this proof, fix $\defnm{f}:\mathbb{N}\rightarrow \mathbb{N}$ to be the function $f_{\ref{thm:CylindricalWallControlledByTangle}}$ given by \cref{thm:CylindricalWallControlledByTangle} and $\defnm{\rho}:\mathbb{N}\times \mathbb{N}\rightarrow \mathbb{N}$ as well as $\defnm{\alpha}:\mathbb{N}\rightarrow \mathbb{N}$ to be the functions $\rho_{\ref{thm:flatwall}}, \alpha_{\ref{thm:flatwall}}$ given by \cref{thm:flatwall}.

Let $\defnm{g}:\mathbb{N}\rightarrow \mathbb{N}$ be any function satisfying the following three properties:
\begin{itemize}
    \item For every $k\in \mathbb{N}, k\ge 2$, we have $g(k)\ge 4\cdot\max\{g(k-1),k\}$. 
    \item For every $k\in \mathbb{N}, k\ge 2$, we have 
    $$g(k-1)>f\left(\max\left\{\rho\left(3k+2,\frac{k^2+3k}{2}\right),\alpha\left(\frac{k^2+3k}{2}\right)+1\right\}\right).$$
    \item For every $k\in \mathbb{N}, k\ge 2$, we have $g(k)\ge \alpha\left(\frac{k^2+3k}{2}\right)+7k-5$.
\end{itemize}
Clearly, such a function $g$ exists. We now claim that this function satisfies the statement of the theorem, i.e.\ for every $k\in \mathbb{N}$ it holds that every digraph $D$ with $\delta^+(D)\ge g(k)$ contains $k$ vertex-disjoint directed cycles of distinct lengths.

Suppose towards a contradiction that this statement does not hold, and let $\defnm{k} \in \mathbb{N}$ be chosen minimal such that it fails. Let \defn{$D$} be some digraph certifying this, i.e.\ such that $\delta^+(D)\ge g(k)$, but such that $D$ does not contain any collection of $k$ vertex-disjoint directed cycles of distinct lengths. Observe that we must have $k\ge 2$: Trivially, $\delta^+(D)\ge g(k)\ge 1$ implies that $D$ contains a directed cycle, so we cannot have $k=1$. 

\begin{claim}\label{claim:nodisjointhighdegrees}
    There exist no two vertex-disjoint non-empty subdigraphs $D_1, D_2$ of $D$ such that $\delta^+(D_1)\ge g(k-1)$ and $\delta^+(D_2)\ge k$.
\end{claim}

\begin{claimproof}
    Suppose towards a contradiction that such subdigraphs $D_1, D_2$ existed. By our minimal choice of~$k$, and since $k\ge 2$, every digraph of minimum out-degree at least $g(k-1)$ contains $(k-1)$ vertex-disjoint directed cycles of distinct lengths. In particular, there exist $(k-1)$ vertex-disjoint directed cycles $C_1,\ldots,C_{k-1}$ in $D_1$ whose lengths are pairwise distinct. Furthermore, by~\cref{lem:trains}, there exist $k$ directed cycles in $D_2$ with pairwise distinct lengths (not necessarily disjoint). In particular, one of these $k$ cycles must have a length that differs from all lengths of cycles in $\{C_1,\ldots,C_{k-1}\}$. Let $C_k$ be such a directed cycle in $D_2$. Since $D_1, D_2$ are vertex-disjoint, $C_1,\ldots,C_k$ forms a collection of $k$ vertex-disjoint directed cycles in~$D$ of distinct lengths, contradicting our initial assumptions about $D$. This concludes the proof of the claim. 
\end{claimproof}

In the remainder of the proof, we set $\defnm{d} := 4\max\{g(k-1),k\}$ and denote by \defn{$\tau$} the set comprising
\begin{itemize}
    \item $(A,B)$ for all proper separations $\{A,B\}$ of $D$ of order $<d/4$ such that no edge crosses from $B$ to $A$, and
    \item $(A,V(D))$ for all $A\subseteq V(D)$ of size $<d/4$.
\end{itemize}
We next observe the following immediate consequence of~\cref{lem:DefaultOrientationIsATangle}.
\begin{claim}\label{claim:defaultorientationtangle}
$\tau$ is a $(d/4)$-tangle of $D$.
\end{claim}
\begin{claimproof}
\cref{claim:nodisjointhighdegrees} and the definition of $d$ imply that there are no two vertex-disjoint subdigraphs of $D$ with minimum out-degree at least $d/4$. Our choice of $g$ and the definition of $d$ imply further that $\delta^+(D)\ge g(k)\ge d$. The claim now follows from~\cref{lem:DefaultOrientationIsATangle}.
\end{claimproof}

We next use~\cref{thm:CylindricalWallControlledByTangle} to obtain a large cylindrical wall in $D$ which is controlled by $\tau$.

\begin{claim}\label{claim:wall}
    There exists a cylindrical wall $W$ of order at least $\rho\left(3k+2,\frac{k^2+3k}{2}\right)$ in $D$ such that for every proper directed separation $\{X,Y\}$ of $D$ of order at most $\alpha\left(\frac{k^2+3k}{2}\right)$ such that no edges cross from $Y$ to $X$ in $D$, there exist a row and a column of $W$ that are contained in $D[Y\setminus X].$
\end{claim}

\begin{claimproof}
    Set $k':=\max\left\{\rho\left(3k+2,\frac{k^2+3k}{2}\right),\alpha\left(\frac{k^2+3k}{2}\right)+1\right\}$. 

    By our choice of $g$, we have that $f(k')<g(k-1)\leq d/4$. Hence, we may apply~\cref{thm:CylindricalWallControlledByTangle} to the restriction of $\tau$ to separations of order less than $f(k')$. We thus find that there exists a cylindrical wall $W$ of order $3k'-2\ge k'\ge \rho\left(3k+2,\frac{k^2+3k}{2}\right)$ in $D$ such that for every proper directed separation $\{X,Y\}$ of $D$ of order at most $\alpha\left(\frac{k^2+3k}{2}\right)$ (and thus, in particular, less than $k'$), such that no edges in $D$ cross from $Y$ to $X$, there exist a row and a column of~$W$ that are contained in $D[Y\setminus X]$. This establishes the assertion of the claim.
\end{claimproof}

Next, we combine the local Directed Flat Wall Theorem (\cref{thm:flatwall}) with~\cref{lem:butterfly} to find a weakly flat subwall of $W$.

\begin{claim}\label{claim:subwall}
    There exists a set of vertices $A\subseteq V(D)$ with $|A|\le \alpha\left(\frac{k^2+3k}{2}\right)$ and a cylindrical subwall $W'\subseteq W$ of order $3k+2$ such that $W'$ is weakly flat in $D-A$.
\end{claim}

\begin{claimproof}
    By \cref{claim:wall}, the cylindrical wall $W$ in $D$ has order at least $\rho\left(3k+2,\frac{k^2+3k}{2}\right)$. \cref{thm:flatwall} now implies that either $D$ contains $\bivec{K}_t$ as a butterfly minor, where $t:=\frac{k^2+3k}{2}$, or that a set~$A$ of vertices and a cylindrical subwall~$W'$ of~$W$ with the desired properties stated in the claim exist. However, the first case is impossible since it would, by~\cref{lem:butterfly}, imply that $D$ contains $k$ vertex-disjoint directed cycles of distinct lengths, contradicting our initial assumptions on~$D$. Hence, the second case has to hold, and hence the set~$A$ and cylindrical subwall~$W'$ with the desired properties indeed exist, establishing the claim. 
\end{claimproof}

In the remainder of this proof, let us define \defn{$R$} as the set of all vertices which are reachable from $V(W')$ in $D-A$. In other words, these are all the vertices in $W'$ together with all vertices that are endpoints of directed paths in $D-A$ starting in some vertex of $W'$. Let us observe that every vertex $v\in R$ satisfies that $N_D^+(v)\setminus A\subseteq R$. This immediately implies that the subdigraph $\defnm{D'} := D[R]$ induced by $R$ has minimum out-degree at least $\delta^+(D)-|A|\ge g(k)-\alpha\left(\frac{k^2+3k}{2}\right)\ge 7k-5$, where the last inequality follows from our initial choice of the function $g$.

Next, we use \cref{lem:walltrains} to prove that the digraph $D'$ cannot be strongly connected.
\begin{claim}\label{claim:disconnected}
    $D'$ is not strongly connected.
\end{claim}

\begin{claimproof}
    Towards a contradiction, suppose that $D'$ is strongly connected. Note that $W' \subseteq D'$ by definition of $R$, and recall that $W'\subseteq D-A$ is a cylindrical wall of order $3k+2$ which is weakly flat in $D-A$. As $D' \subseteq D-A$, it follows that $W'$ is also weakly flat in~$D'$. Since $D'$ is strongly connected and satisfies $\delta^+(D')\ge 7k-5$ as argued above, we may now apply \cref{lem:walltrains} to $D'$ and the wall $W'$ contained in it, and find that $D'$ (and thus $D$) contains $k$ vertex-disjoint directed cycles of pairwise distinct lengths. This contradicts our initial assumptions on $D$, and hence $D'$ cannot be strongly connected.
\end{claimproof}

Finally, we combine the previous claims to obtain the desired final contradiction, which then concludes the proof of the theorem. By~\cref{claim:disconnected}, there exists a partition $S\; \dot\cup\; T = R$ of the vertex set~$R$ of $D'=D[R] \subseteq D-A$ into two non-empty sets~\defn{$S$} and~\defn{$T$} such that there are no edges in~$D$ starting in~$T$ and ending in~$S$. Let us now define $\defnm{Y} := T\cup A$ and $\defnm{X} := V(D) \setminus T =  S\cup (V(D)\setminus R)$. 
\smallskip

We claim that $\{X,Y\}$ is a directed separation of~$D$. Clearly, $X\cup Y=V(D)$, and so it suffices to show that no edge of~$D$ crosses from $Y$ to~$X$. Indeed, suppose towards a contradiction that there existed some edge $e=(u,v)\in E(D)$ with $u\in Y\setminus X=T$ and $v\in X\setminus Y=S\cup (V(D)\setminus (R\cup A))$. Then, since by assumption no edge of $D$ starts in $T$ and ends in $S$, we must have $v\in V(D)\setminus (R\cup A)$. However, since $u\in T\subseteq R$ we have $v\in N_D^+(u)\setminus A\subseteq R$ (by definition of $R$), a contradiction to $v\in V(D)\setminus (R\cup A)$. 

Hence, no edge crosses from $Y$ to $X$, and $\{X,Y\}$ is a directed separation of $D$. 
Furthermore, observe that $X\cap Y=A$ and that $\emptyset\neq S\subseteq X\setminus Y, \emptyset\neq T\subseteq Y\setminus X$, so in fact $\{X,Y\}$ is a proper separation. 
\smallskip

Since $W'$, as a cylindrical wall, is strongly connected, and because $W' \subseteq D-A = D - (X \cap Y)$, we must have either $W' \subseteq D[Y\setminus X]$ or $W' \subseteq D[X\setminus Y]$.
We next show that in fact $W' \subseteq D[Y\setminus X]$. For this, observe that since $X \cap Y = A$, the order of the separation $\{X,Y\}$ is $|A|\le \alpha\left(\frac{k^2+3k}{2}\right)$. We may thus apply \cref{claim:wall} to $\{X,Y\}$ to find that there exists a row and a column of $W$ that are both contained in $D[Y\setminus X]$. By \cref{obs:wallintersection}, it follows that $V(W') \cap (Y\setminus X) \neq \emptyset$, which by our earlier observation implies that $W' \subseteq D[Y\setminus X]$.
\smallskip

We now arrive at our final contradiction. For this, recall that $S$ is non-empty, and consider any vertex $s\in S$. Then $s\in R$, and therefore, by the definition of $R$, there exists a directed path in $D-A$ starting in $V(W')\subseteq Y\setminus X$ and ending in $s \in S \subseteq X \setminus Y$. But this path must use an edge from $Y\setminus X$ to $X\setminus Y$ (as it avoids $A = X \cap Y$), contradicting that no edge of $D$ crosses from $Y$ to $X$ as shown above.

This is the desired contradiction, which shows that our initial assumption in this proof, namely that there exist a number $k\in \mathbb{N}$ and a digraph $D$ with $\delta^+(D)\ge g(k)$ and no $k$ vertex-disjoint directed cycles of distinct lengths, was false. We thereby conclude the proof of the theorem. 
\end{proof}
\section{Weighted generalisations}\label{sec:weighted}
In this section, we supply the proofs of the weighted consequences of~\cref{main:CyclesDistinctLengths}, namely~\cref{cor:vertexweights,cor:edgeweights}. We start with the proof of~\cref{cor:vertexweights}, from which~\cref{cor:edgeweights} can then be swiftly deduced.

\begin{proof}[Proof of~\cref{cor:vertexweights}]
Let $k\in\mathbb{N}$, a digraph $D$ with $\delta^+(D)\ge g(k)$, and a strictly positive vertex-weighting $w:V(D)\rightarrow \mathbb{R}_+$ be given to us. Our goal is to show that there exists a collection $C_1,\ldots,C_k$ of $k$ vertex-disjoint directed cycles in $D$ such that for all distinct $i,j\in [k]$, we have $\sum_{v\in V(C_i)}w(v)\neq \sum_{v\in V(C_j)}w(v)$.

In what follows, we first prove this under the additional assumption that $w$ is integral, i.e.\ $w(v)\in \mathbb{N}$ for all $v\in V(D)$. After that, we come back to the general case (which we will reduce to the integral case).
\medskip

So suppose that $w$ is integral. Now, form an auxiliary digraph \defn{$D_w$} as follows. All vertices of $D$ will also be vertices of $D_w$. Then, for every vertex $v\in V(D)$ with $w(v)\ge 2$, we add  precisely $w(v)-1$ pairwise disjoint sets $V_1^v,\ldots,V_{w(v)-1}^v$ of fresh vertices (disjoint from $V(D)$ and from any previously added vertices), where $|V_1^v|=\cdots=|V_{w(v)-1}^v|=g(k)$. Having done this for all $v\in V(D)$ with $w(v)\ge 2$, we have defined all the vertices of~$D_w$. We call the vertices in $V(D)$ the \defn{original} vertices and the vertices in $V(D_w)\setminus V(D)$ the \defn{new} vertices. Next, we define the edges of~$D_w$:

For every vertex $v\in V(D)$ with $w(v)=1$, we have the same edges as in the original digraph, i.e.\ we add all the edges $(v,u)$, for $u \in N_D^+(v)$, to~$D_w$. For all $v\in V(D)$ with $w(v)\ge 2$, we instead add the following edges: $(v,u)$ for every $u\in V_1^v$, then for all $1\le i<w(v)-1$ (if any) we add all ordered pairs in $V_i^v\times V_{i+1}^v$ as edges, and finally all the ordered pairs in $V_{w(v)-1}^v\times N_D^+(v)$. This finishes the description of the edges of~$D_w$.\looseness=-1
\smallskip

Observe that since $\delta^+(D)\ge g(k)$, the digraph~$D_w$ defined in this way has minimum out-degree at least~$g(k)$. Hence, we may apply~\cref{main:CyclesDistinctLengths} to it and obtain a collection $C_1,\ldots,C_k$ of $k$ vertex-disjoint directed cycles in $D_w$ of pairwise distinct lengths.

Now fix any $t\in [k]$. Let $\ell_t$ denote the length of $C_t$, and $u_0u_1u_2\cdots u_{\ell_t}=u_0$ the cyclic sequence of vertices along $C_t$ such that $(u_{r-1},u_r)\in E(D_w)$ for every $r\in [\ell_t]$. Note that $D_w-V(D)$ is acyclic, and hence $C_t$ must use at least one original vertex. Let $r_1<\cdots<r_s$ for some $s\ge 1$ be the complete enumeration of all indices $r$ such that $u_r$ is an original vertex. It then follows directly from the definition of $D_w$ that for every $i\in [s-1]$, we have that $u_{r_{i+1}}\in N_D^+(u_{r_i})$. Similarly, $u_{r_1}\in N_D^+(u_{r_s})$. Hence, the cyclic vertex-sequence $u_{r_1}\ldots u_{r_s}u_{r_1}$ defines a directed cycle in $D$, which we denote as $C_t'$. It can furthermore be easily seen from the definition of $D_w$ that the segment of $C_t$ from $u_{r_i}$ to $u_{r_{i+1}}$ has length exactly $w(u_{r_i})$, for every $1\le i\le s-1$, and analogously the segment of $C_t$ from $u_{r_s}$ to $u_{r_1}$ has length exactly $w(u_{r_s})$. Since these segments disjointly decompose the edges of $C_t$, we find that $\ell_t=\sum_{i=1}^{s}w(u_{r_i})=\sum_{v\in V(C_t')}w(v)$. 

Since by choice of $C_1,\ldots,C_k$ we have that $\ell_1,\ldots,\ell_k$ are pairwise distinct, it follows that the total weights of $C_1',\ldots,C_k'$ according to $w$ are pairwise distinct, as well. Furthermore, since $V(C_t')\subseteq V(C_t)$ for every $t\in [k]$ by definition and since $C_1,\ldots,C_k$ are vertex-disjoint, it follows that also the directed cycles $C_1',\ldots,C_k'$ in $D$ are pairwise disjoint, and hence form a cycle collection with the desired properties in $D$. This concludes the proof in the case that $w$ is integral.
\medskip

Now, let us go back to the general case $w:V(D)\rightarrow\mathbb{R}_+$. Suppose towards a contradiction that there do not exist $k$ vertex-disjoint directed cycles in $D$ with pairwise distinct weights according to $w$. Possibly by multiplying all weights with a sufficiently large positive scalar, we may without loss of generality assume that $w(v)\ge 1$ for all $v\in V(D)$. Now, let $\mathcal{K}$ denote the set of all (unordered) collections of $k$ vertex-disjoint directed cycles in~$D$. Note that $\mathcal{K}$ is finite. Our assumption about $w$ implies that for every collection $K\in \mathcal{K}$ there exist two distinct cycles $C_K^1,C_K^2\in K$ such that $\sum_{v\in V(C_K^1)}w(v)=\sum_{v\in V(C_K^2)}w(v)$. 

Hence, the (finite) system of linear equalities and inequalities

\begin{align*}
\sum_{v\in V(C_K^1)}x_v-\sum_{v\in V(C_K^2)}x_v&=0,\,\,\,\,\,\forall K\in \mathcal{K},
\\
    x_v&\ge 1,\,\,\,\,\forall v\in V(D),
\end{align*}
is feasible. Note that all the coefficients and the right hand side of these inequalities are rational (in fact integers). Hence, by standard polyhedral theory there must also exist a \emph{rational} feasible solution of the system. Take such a rational solution and multiply all its entries by a positive common multiple of all the denominators of the fractions representing its entries. Since scaling a solution to the above system by a positive integer maintains feasibility, it follows that the system in fact has an integral solution. In other words, there exists a strictly positive integer-weighting $w':V(D)\rightarrow \mathbb{N}$ of $D$ such that $$\sum_{v\in V(C_K^1)}w'(v)=\sum_{v\in V(C_K^2)}w'(v)$$ for all $K\in \mathcal{K}$. This immediately implies that there does not exist a collection of $k$ vertex-disjoint directed cycles in $D$ that have pairwise distinct weights with respect to $w'$. However, this contradicts the fact that we already proved the assertion of the corollary for the case of integral weightings. This contradiction shows that our initial assumption about the weighting $w:V(D)\rightarrow \mathbb{R}_+$ was false. This concludes the proof of the corollary.
\end{proof}

Finally, we deduce~\cref{cor:edgeweights}.

\begin{proof}[Proof of~\cref{cor:edgeweights}]
Let $k,m\in \mathbb{N}$, a digraph $D$ with $\delta^+(D)\ge g(k)\cdot m$, a set $S\subseteq \mathbb{R}_+$ with $|S|=m$, and an edge-weighting $w:E(D)\rightarrow S$ be given to us. We must show that there exist $k$ vertex-disjoint directed cycles in $D$ with pairwise distinct total edge-weights according to $w$. 

For this, define a spanning subdigraph $D'$ of $D$ as follows. For every vertex $v\in V(D)$, note that since $g(k)\cdot m\le d^+(v)=\sum_{s\in S}|\{(v,u)\in E(D):w(v,u)=s\}|$ and $|S|=m$, by the pigeon hole principle there must exist some $s_v\in S$ such that at least~$g(k)$ arcs starting in~$v$ have weight~$s_v$. We then keep only those arcs in~$D'$ and discard all others.

It is clear that the spanning subdigraph $D'\subseteq D$ defined in this way has minimum out-degree at least~$g(k)$. Furthermore, observe that if we define the strictly positive vertex-weighting~$w'$ on $D'$ by $w'(v):=s_v$ for every $v\in V(D')=V(D)$, we have that the total edge-weight of any directed cycle in~$D'$ with respect to~$w$ equals the total vertex-weight of the same cycle with respect to~$w'$. Hence, when we apply~\cref{cor:vertexweights} to~$D'$ to obtain $k$ vertex-disjoint directed cycles in~$D'$ with pairwise distinct total vertex-weights according to~$w'$, these will also have pairwise distinct edge-weights according to~$w$, while also forming a collection of $k$ vertex-disjoint directed cycles in~$D$. This concludes the proof of the corollary.
\end{proof}

\section{Concluding remarks} \label{sec:conc}

We conclude with the natural open problem of determining the smallest possible function $g:\mathbb{N}\rightarrow \mathbb{N}$ for which~\cref{main:CyclesDistinctLengths} holds. The only lower bound we are aware of is $\frac{k^2+3k-2}{2}$, which can be obtained by observing that every digraph containing $k$ vertex-disjoint directed cycles of distinct lengths must contain at least $2+3+\cdots+k+(k+1)=\frac{k^2+3k}{2}$ vertices, and to enforce this one needs to require at least an out-degree of $\frac{k^2+3k}{2}-1=\frac{k^2+3k-2}{2}$, as otherwise a complete digraph of suitable size would form a counterexample to the statement (see also \cite{bensmail}).

\nopagebreak[4]
\begin{problem}
    What is the smallest function $g:\mathbb{N}\rightarrow \mathbb{N}$ such that every digraph of minimum out-degree at least $g(k)$ contains $k$ vertex-disjoint directed cycles of distinct lengths?
\end{problem}

\section*{Statement on the use of AI}

The whole paper has been written entirely by the human authors, and no AI was used in the writing process. Concerning the generation of the mathematical ideas, all of the presented proof ideas are fully due to the authors, with the sole exception of the proof presented in \cref{subsec:CylindWallControlledByTangle}, for which we used assistance from ChatGPT 5.5 Pro. Concretely, ChatGPT suggested to use \cref{thm:GridTheorem:WellLinkedSet} and \cref{lem:BrambleCoverIsWellLinked}, and it proved \cref{lem:BrambleCoverIsWellLinked}.

\section*{Acknowledgments}

We dearly thank Dan Kr\'{a}l, Filip Kučerák, and Lina Simbaqueba for inspiring and helpful discussions. In particular, we thank Dan Kr\'{a}l for asking a good question that inspired us to prove~\cref{cor:vertexweights,cor:edgeweights}. We also thank Irene Muzi for insights into the Directed Grid Theorem, and Sebastian Wiederrecht for insights into the Directed Grid Theorem and the Directed Flat Wall Theorem. The second author would like to thank Dan Kr\'{a}l and Leipzig University for the hospitality during a research stay, where the presented research was initiated.

\printbibliography

\end{document}